\documentclass[numbook, envcountsect, envcountsame, envcountreset, runningheads, smallextended]{svjour3} 
\smartqed  
\usepackage{graphicx}
\usepackage{amssymb,amsbsy}
\usepackage{color}
\usepackage{xcolor}
\usepackage{hyperref}

 \journalname{Applied Mathematics and Optimization}
\begin{document}

\begin{center} {\bf \large On   Approximate and Weak Correlated   
Equilibria in  Constrained Discounted Stochastic Games }
\end{center}
\begin{center}{\bf    Anna Ja\'skiewicz$^{a}$, Andrzej S. Nowak$^{b}$ } \end{center}
\noindent 
\noindent$^{a}$Faculty of Pure and Applied Mathematics, Wroc{\l}aw University of Science and Technology,  Wroc{\l}aw, Poland \\
{\footnotesize {\it email: anna.jaskiewicz@pwr.edu.pl}}\\
\noindent$^{b}$Faculty of Mathematics, Computer Science, 
and Econometrics,  University of Zielona G\'ora,  Zielona G\'ora, Poland \\
{\footnotesize {\it email: a.nowak@wmie.uz.zgora.pl}}\\

\begin{center}
\today
\end{center}
 
\noindent{\bf Abstract.} 
In this paper, we consider  constrained discounted  stochastic  games with a countably generated  state space
and norm continuous transition probability having a density function. 
We prove existence of approximate stationary equilibria and stationary weak correlated equilibria. 
Our results imply the existence of stationary Nash equilibrium in $ARAT$ stochastic games.\\



Running head: Approximate and Weak Correlated Equilibria 

\section{Introduction}

Constrained Markov decision processes and stochastic games have numerous applications 
in operations research, economics, computer sciences, consult with \cite{a,as,jn3,piu} 
and references cited therein.  They
arise in situations, in which a controller or player has  many objectives. For example,    
when she or he  wants to minimise one type of cost while keeping other
costs lower than some given bounds. Constrained stochastic games with finite state and 
action spaces were first studied by Altman and Shwartz \cite{as}. Their work was extended 
to some classes of games with countable state spaces in \cite{ahl,zhg} by finite state 
approximations. A more direct approach based on properties of measures induced by strategies 
and occupation measures was presented in \cite{jn3}.

In this paper, we study discounted constrained stochastic games  with a general state space 
and the transition probability having a density function. Such two-person games with additive rewards 
and additive transition structure ($ARAT$ games) were recently studied by 
Dufour and Prieto-Rumeau \cite{dpr2}. They established the existence of stationary 
Nash equilibria generalising the result of Himmelberg et al. \cite{hprv} proved 
for unconstrained games. Moreover, their theorem also holds for $N$-person$ ARAT$   games satisfying 
the standard Slater condition. As shown in a highly non-trivial example by Levy \cite{lm}, 
the games under consideration in this paper may have no stationary Nash equilibrium 
in the unconstrained case. It can be seen, that this example applies to the constrained case as well. 
Thus, results on approximate equilibria as in \cite{n1,ww}
became more valuable. They are stated for the unconstrained case, and in this paper 
we extend the main result from \cite{n1} to a class of constrained games. 
In this way, we establish the existence of approximate stationary equilbria for 
discounted stochastic games with constraints and general state spaces.
It should be noted that the existence of stationary equilibria in discounted unconstrained
games was proved only in some special cases, for instance, for $ARAT$ games \cite{hprv} or games 
with transitions having no conditional atoms \cite{hsun}.
For a survey of results on stationary  and non-stationary Nash equilibria the reader 
is referred to \cite{jn1}. 

The other group of papers comprise the ones on stationary equilibria with public 
signals, see \cite{dgmm,hrr,nr}.  Such solutions can be 
viewed as special communication or correlated equilibria widely discussed in dynamic frameworks
(repeated, stochastic or extensive form games) in \cite{f1,f2,my,s,sv}. 
They were inspired by the seminal papers of Aumann \cite{au1,au2}.
A weaker version of correlated equilibrium was proposed by Moulin and Vial \cite{mv}.  
According to their approach 
a correlated strategy in a finite (bimatrix) game is a probability 
distribution $\nu$ on the set of pure strategy pairs. Every player has 
to decide whether to accept $\nu$  or to use his or her individual strategy. 
If player $i$ uses an individual strategy and player $j\not=i$ obeys $\nu$, 
then a pure action for player $j$ is selected by the marginal distribution of $\nu$
on his/her pure actions. Then $\nu$ is an equilibrium, if no unilateral deviations 
from it are profitable. This solution is called a weak correlated equilibrium 
or a correlated equilibrium with no exchange of information \cite{mv}. 
In contrast to Aumann's approach, the players who accepted $\nu$ cannot change 
actions after using the lottery $\nu$. The solution proposed by Moulin and Vial \cite{mv} 
has an interesting property. Namely, the authors constructed a bimatrix game, in which the equilibrium payoffs 
in their equilibrium concept strictly dominate in the Pareto sense the payoffs in Aumann's equilibrium, see \cite{m,mv}.

In \cite{n2} the concept of Moulin and Vial is used to an unconstrained discounted 
stochastic game with a general state space. However, as shown by Solan and Vieille, \cite{sv},
the notion of a  weak correlated equilibrium can be also regarded as a special case  of a correlation scheme. 
In this paper, we extend the result from \cite{n2} to a large class of   
discounted stochastic games with so-called integral constraints. We apply our recent 
result from  \cite{jn3} for games with discrete state spaces and use an approximation technique. 
A stationary weak correlated equilibrium is obtained as a limit (in the weak* sense) 
of approximate equilibria.
Our result generalises the main theorem of Dufour and Prieto-Rumeau \cite{dpr2} given for 
$ARAT$ games, if the action sets for players do not depend on the state. 
We wish to emphasise that the considerations of other classes of correlated equilibria in constrained stochastic 
games (like equilibria with public signals) seem to be very challenging  for many reasons. Firstly,
the integral constraints are difficult to apply. Secondly, the usual methods 
from dynamic programming (Bellman's principle) or backward and forward induction  
used in unconstrained cases are not applicable. Perhaps further possible results can be obtained  
for other correlated equilibria but under different type of constraints.

The paper is organised as follows. The model and main results on equilibria are contained in Section \ref{s2}.
Section \ref{s3} presents the approximation technique and the proofs of two main theorems.
Section \ref{s4} is devoted to the proof
on the existence of a weak correlated equilibrium and a discussion on our assumptions.  
In Section \ref{s5}, we show that
the example given in \cite{lm} can be used to show that discounted
constrained stochastic games studied in this paper may not have
stationary Nash equilibria.  Section \ref{s6} discusses a useful transformation that shows how to easily 
extend our results formulated for bounded cost functions to unbounded ones. 
In Appendix (Section \ref{s7}) we give a crucial lemma on a replacement one strategy by another. 
It is used in the proofs of our main theorems 
on equilibria in constrained stochastic games. 
 
 \section{ The   game model  and main results}
\label{s2}

In this section, we describe constrained discounted stochastic games 
with general state space and our basic assumptions. 
We provide  our main results in three cases. Firstly, we give a theorem 
on the existence of a stationary   approximate equilibrium assuming that the players 
play the game independently.  Secondly, we drop the constraints and give a theorem
 on the existence of  a stationary  $\varepsilon$-equilibrium  
for every initial state, extending the main result in \cite{n1}.   
Finally, we show that the constrained  stochastic games under consideration possess stationary 
weak correlated equilibria introduced in the static (bimatrix) case by Moulin and Vial \cite{mv}.

\subsection{\bf  Approximate Nash equilibria in constrained discounted stochastic games} 

The non-zero-sum {\it constrained stochastic   game}  $(CSG)$ is described by the following objects: 
\begin{itemize}
\item ${\cal N}=\{1,2,...,N\}$  is the {\it set of players}. 
\item $X$ is a  {\it state space} endowed with a countably
generated $\sigma$-algebra $\cal F.$
\item $A_i$ is a compact metric {\it action space} for
player $i\in {\cal N} $  endowed with the Borel $\sigma$-algebra. We put
$$ A:=\prod_{j\in {\cal N}} A_j\quad \mbox{and}\quad
A_{-i}:=\prod_{j\in {\cal N}\setminus\{ i\}}  A_j,$$
$$  
 \mathbb{K}_i:= \{(x,a_i): x\in X,\ a_i\in A_i  \}, \quad 
 \mathbb{K}:= \{(x,\pmb{a}): x\in X,\ \pmb{ a}=(a_1,...,a_n)\in A \}.$$
\item The real-valued functions
$c_i^\ell:\mathbb{K}\to \mathbb{R},$ 
where  $i\in {\cal N},$  $\ell \in {\cal L}_0={\cal L}\cup\{0\} $ 
with  ${\cal L}=\{1,...,L\},$   are product  measurable. Here,
$c_i^0$ is the {\it cost-per-stage  function} for player  $i\in {\cal N},$ 
and for each  $\ell\in {\cal L},$   $c_i^\ell$ is a function used in the definition 
of the $\ell$-th {\it constraint} for this player.  It is assumed that there exists $b>0$ such that
$$|c^\ell_i(x,\pmb{a})|\le b,\quad\mbox{for all}\quad i\in {\cal N},\ 
\ell \in {\cal L}_0,\ (x,\pmb{a})\in \mathbb{K}.$$

\item $p(dy|x,\pmb{a})$ is the transition probability from $x$ to $y\in X,$
when the players choose a profile $\pmb{ a}=(a_1,a_2,...,a_N)$ of actions in $ A.$ 
\item $\eta$ is  the {\it initial state distribution}.
\item $\alpha\in (0,1)$ is the {\it discount factor}.
\item $\kappa_i^\ell$ are constraint constants, $i\in {\cal N},$ $\ell \in {\cal L}.$ 
\end{itemize}
Let $\mathbb{N}=\{1,2,...\} .$ Define $H^1=X$ and $H^{t+1}=   \mathbb{K}\times H^{t}$  for $t\in \mathbb{N}.$ 
An element  $h^t=(x^1,\pmb{ a}^1,\ldots,x^t)$ of $H^t$ represents a history of the game up to the $t$-th
period, where  $\pmb{ a}^k=(a^k_1,\ldots,a^k_N)$ is the profile of actions chosen by the players in the state $x^k$ 
on the $k$-th stage of the game,  $h^1=x^1.$

Strategies for the players are defined in the usual way. 
A {\it strategy} for player $i\in{\cal N}$ is a sequence $\pi_i=(\pi_{i}^t)_{t\in  \mathbb{N}},$ 
where each $\pi_{i}^t$  is a transition probability from $H^t$ to $A_i.$  
By $\Pi_i$ we denote the {\it set of all strategies} for player $i.$
Let  $\Phi_i$ be the set of transition probabilities from $X$ to $A_i.$  
A {\it stationary strategy} for player $i$ is a constant
sequence $ (\varphi_{i}^t)_{t\in  \mathbb{N}},$ where $\varphi_i^t=\varphi_i$ 
for all $t\in\mathbb{N}$ and some $\varphi_i\in\Phi_i.$ 
Furthermore, we shall identify  a stationary strategy for player $i$  
with the  constant element $\varphi_i$  of the sequence. 
Thus, the {\it set of all    stationary } strategies of player $i$  is also denoted by $\Phi_i.$
We define
$$\Pi = \prod_{i=1}^N \Pi_i \quad\mbox{and}\quad
\Phi = \prod_{i=1}^N \Phi_i.$$
Hence, $\Pi$ ($\Phi$) is the set of all (stationary) multi-strategies of the players.

Let  $H^\infty = \mathbb{K}\times \mathbb{K}\times \cdots$
be the space of all infinite histories of the game endowed with
the  product $\sigma$-algebra. For any  multi-strategy  $\pmb{ \pi}\in\Pi$,
a unique probability measure $\mathbb{P}_\eta^{\pmb{\pi}}$  and
 a stochastic process $(x^t,\pmb{a}^t)_{t\in \mathbb{N}}$  
are defined on $H^\infty$ in a canonical way, see 
the Ionescu-Tulcea theorem, e.g., Proposition V.1.1 in \cite{n}. 
The measure $\mathbb{P}_\eta^{\pmb{\pi}}$  is induced by $\pmb{\pi},$ 
the transition probability $p$ and the initial distribution $\eta.$  
The expectation operator with respect to $\mathbb{P}_\eta^{\pmb{\pi}}$ 
is denoted by $\mathbb{E}_\eta^{\pmb{\pi}}.$ 

Let $\pmb{\pi}\in\Pi$ be any multi-strategy. For each
$i\in {\cal N}$ and $\ell \in {\cal L}_0$,  the {\it discounted cost functionals} are defined as
$$
J_i^\ell(\pmb{\pi}) = (1-\alpha)\mathbb{E}^{\pmb{\pi}}_\eta\left[\sum_{t=1}^\infty 
\alpha^{t-1} c_i^\ell(x^t,\pmb{a}^t) \right].
$$
 
We assume that $J^0_i(\pmb{\pi})$ is the expected discounted cost of player $i\in {\cal N}$,  
who wishes to  minimise it over $\pi_i \in\Pi_i$ in such a way  that
the following constraints are satisfied
$$
J^\ell_i(\pmb{\pi}) \le \kappa^\ell_i \quad\mbox{for all}\quad
\ell\in {\cal L}.
$$
A multi-strategy $\pmb{\pi}$ is {\it feasible}, if the above inequality
holds for each $i\in {\cal N},$  $\ell\in {\cal L}.$ We denote by $\Delta$ the 
set of all feasible multi-strategies in the $CSG.$

As usual, for any $\pmb{\pi}\in\Pi$,  we denote by  $\pmb{\pi_{-i}}$
the multi-strategy of all players but player $i,$ that is,
$ \pmb{\pi_{-1}} =(\pi_2,...,\pi_N),$  $ \pmb{\pi_{-N}} =(\pi_1,...,\pi_{N-1}),$
and for $i\in {\cal N}\setminus \{1,N\},$  
 $$\pmb{\pi_{-i}}
=(\pi_1,\ldots,\pi_{i-1},\pi_{i+1},\ldots,\pi_N).$$
We identify $[\pmb{\pi_{-i}},\pi_i]$ with $\pmb{\pi}.$ 
For each $\pmb{\pi}\in\Pi$, we define the set of feasible 
strategies for player $i$ with $\pmb{\pi_{-i}}$ as
$$\Delta_i(\pmb{\pi_{-i}})=\{\pi_i\in\Pi_i: \ J^\ell_i(\pmb{\pi}) 
=  J^\ell_i([\pmb{\pi_{-i}},\pi_i]) 
\le \kappa^\ell_i \quad\mbox{for all}\quad
\ell\in {\cal L}\}.$$

Let $\pmb{\pi}=(\pi_1,\pi_2,...,\pi_N) \in\Pi$ and $\sigma_i\in\Pi_i.$   
By $[\pmb{\pi_{-i}},\sigma_i]$ 
we denote the multi-strategy, 
where player $i$ uses $\sigma_i$ and every player $j\not= i$ uses $\pi_j.$

\begin{definition}
\label{def1}
A multi-strategy $\pmb{\pi}^*\in\Pi$ is an  {\it approximate equilibrium} 
in the $CSG$ (for given $\varepsilon >0$), if 
for every $i\in {\cal N}$ and $\ell\in {\cal L},$    
\begin{equation}
\label{def1e1}
J^\ell_i(\pmb{\pi}^*) \le \kappa^\ell_i +\varepsilon,
\end{equation}
 and for every $i\in {\cal N},$ 
\begin{equation}
\label{def1e2}
J^0_i(\pmb{\pi}^*)-\varepsilon \le\inf_{\sigma_i \in\Delta_i(\pmb{\pi_{-i}^*})} J^0_i([\pmb{\pi_{-i}^*},\sigma_i]).
\end{equation} 
A multi-strategy $\pmb{\pi}^*\in\Pi$ is an  $\varepsilon$-{\it equilibrium} 
in the $CSG$ (for given $\varepsilon \ge 0$),  if  
(\ref{def1e2}) holds  and $J^\ell_i(\pmb{\pi}^*) \le \kappa^\ell_i $ for every $i\in {\cal N}$ and $\ell\in {\cal L}.$   
A $0$-equilibrium is called a Nash equilibrium in the $CSG.$
\end{definition}

Note that,   every $\varepsilon$-equilibrium  is approximate, but not vice versa. 
For small $\varepsilon>0,$ condition (\ref{def1e1}) allows for a slight violation of the feasibility of  $\pmb{\pi}^*$. 
Further comments on this condition  the reader will find in
Remark \ref{appreq}.\\
 
We now formulate our basic assumptions.\\

\noindent{\bf  Assumption A1} \\  
 The functions $c^\ell_i(x,\cdot)$ are continuous on $A $ 
for all $x\in X,$ $i\in {\cal N}$ and $\ell\in {\cal L}_0.$
\\

\noindent{\bf  Assumption A2} \\  
The transition probability $p$ is of the form
$$p(B|x,\pmb{a})=\int_B \delta(x,y,\pmb{a})\mu(dy), \quad B\in {\cal F},
$$
where $\mu$ is a probability measure on $\cal F$ and $\delta$ is a product measurable 
non-negative  (density) function such that, if $\pmb{a}^n \to \pmb{a}$ as  $n\to \infty,$ then 
$$
\int_X|\delta(x,y,\pmb{a}^n) - \delta(x,y,\pmb{a})|\mu(dy) \to 0.
$$\\

This assumption means the norm continuity of $p$ with respect to action profiles.\\

\noindent{\bf  Assumption A3} \\   For each stationary multi-strategy
$\pmb{\varphi}\in \Phi$ and for each player $i\in {\cal N},$ there exists
$\pi_i\in \Pi_i$ such that 
$$
J^\ell_i([\pmb{\varphi_{-i}},\pi_i]) \le \kappa^\ell_i  \quad \mbox{for all}\quad \ell \in {\cal L}.
$$
\\

Assumption  {\bf A3} is standard in the theory of constrained decision processes and 
stochastic games \cite{a,as,dpr2,jn3}.

\begin{remark}\label{rS}
From Assumption {\bf A3},  Lemma 2.3 in \cite{dpr2} and Lemma 24 in
 \cite{piu} it follows that the strategy $\pi_i\in \Pi_i$ can be replaced a stationary strategy 
 $\sigma_i\in\Phi_i$ such that 
 $$
J^\ell_i([\pmb{\varphi_{-i}},\pi_i])= 
J^\ell_i([\pmb{\varphi_{-i}},\sigma_i])   \quad \mbox{for all}\quad \ell \in {\cal L}.$$
The proof of Lemma 24 in \cite{piu} on the equivalence of these
strategies is formulated  for  models with Borel state spaces.  However, it is also valid in our framework
(see pages 307-309 in \cite{piu}) with the exception that we need an appropriate
disintegration result. In this matter,  consult with Lemma 2.3 in \cite{dpr2} or Theorem 3.2 in \cite{fg}.
\end{remark}

We are ready to state our first  main result.  

\begin{theorem}
\label{thm1}
 Assume {\bf A1}, {\bf A2} and {\bf A3}. Then, for each $\varepsilon >0,$ the $CSG$ 
 possesses a stationary approximate  equilibrium.
\end{theorem}

\begin{remark}
\label{appreq}
The proof of this result is given in Section \ref{s3}. 
We prove that a  stationary  approximate  equilibrium for given $\varepsilon >0$
consists of strategies that are  piecewise constant functions of the state variable. 
We observe that, under assumptions of Theorem \ref{thm1}, 
condition (\ref{def1e1}) with $\varepsilon=0$ need not be satisfied by piecewise constant stationary multi-strategies.
Therefore, the existence of an $\varepsilon$-equilibrium  in the $CSG$ is an open issue.
We would like to emphasise that Theorem \ref{thm1} is crucial in our proof of
Theorem \ref{thm3} on weak correlated equilibria, where we apply an asymptotic approach when $\varepsilon \to 0.$ 
\end{remark}

\begin{remark}
\label{r1}
The only result in the literature on the existence 
of stationary Nash equilibria in $CSGs$ with general state space 
was given by Dufour and Prieto-Rumeau \cite{dpr2}. 
It concerns so-called discounted  additive rewards and additive transition ($ARAT$) 
stochastic games. 
 In the two-person case the $ARAT$ assumption means that $c_i^\ell(x,a_1,a_2)=
c_{1i}^\ell(x,a_1) + c_{2i}^\ell(x, a_2)$ and $p(\cdot|x,a_1,a_2)= 
p_1(\cdot|x,a_1 )+ p_2(\cdot|x, a_2), $  where $p_1$ and $p_2$ are transition subprobabilities. 
The results  in \cite{dpr2}  are given for two-person games
satisfying the standard Slater condition (Assumption {\bf A3} with strict inequalities). 
However, they can be easily extended by the same methods to $N$-person $ARAT$ stochastic  games.    
A simple adaptation of the counterexample by Levy and McLennan \cite{lm} 
given for unconstrained discounted stochastic 
games implies that stationary Nash equilibria may not exist 
in the constrained stochastic games studied in this paper. 
For more details see Section \ref{s5}.
\end{remark}  

\begin{remark}
\label{rSlatter}
We wish to emphasise that the Slater condition is not needed 
for the establishing an  approximate equilibrium in $CSGs.$
\end{remark}  

\subsection{\bf An update on stationary approximate equilibria in unconstrained discounted stochastic games}

In this subsection, we drop the constraints.
By the Ionescu-Tulcea theorem \cite{n}, any multi-strategy
$\pmb{\pi}\in \Pi$ and any initial state $x\in X,$ 
induce  a unique probability measure 
$\mathbb{P}_x^{\pmb{\pi}}$ on $H^\infty.$  
The expectation operator with respect to 
$\mathbb{P}_x^{\pmb{\pi}}$ is denoted by $\mathbb{E}_x^{\pmb{\pi}}.$ 

The {\it discounted cost } for player $i\in {\cal N}$   is defined as
$$
J_i^0(\pmb{\pi})(x) = (1-\alpha)\mathbb{E}^{\pmb{\pi}}_x\left[\sum_{t=1}^\infty\alpha^{t-1} c_i^0(x^t,\pmb{a}^t) \right].
$$

\begin{definition}
\label{def2}
 Let $\varepsilon \ge 0$ be fixed. 
A multi-strategy $\pmb{\pi}^*\in\Pi$ is an {\it $\varepsilon$-equilibrium}  in the unconstrained discounted stochastic game, if 
$$J^0_i(\pmb{\pi}^*)-\varepsilon \le\inf_{\sigma_i \in\Pi_i} 
J^0_i([\pmb{\pi}^*,\sigma_i])$$ 
for every player $i\in {\cal N}$
 and for all initial states $x\in X.$
A $0$-equilibrium is called a Nash equilibrium.
\end{definition}

\begin{theorem} 
\label{thm2}   Under assumptions {\bf A1}  {\it  and} {\bf A2},  
for any $\varepsilon >0,$
the unconstrained discounted stochastic game has a stationary  $\varepsilon$-equilibrium.
\end{theorem}

The proof is given in Section \ref{s3}. 

\begin{remark}
\label{r2}
Stationary Nash equilibria exist only in some special cases of  
stochastic games satisfying assumptions {\bf A1} and {\bf A2}, see 
\cite{hprv} ($ARAT$ games), \cite{hsun} (other classes of games) and \cite{jn1} (a survey).
As shown by Levy and McLennan \cite{lm} stationary Nash equilibria need not exist in 
general under assumptions of Theorem \ref{thm2}.
\end{remark}

\begin{remark}
\label{r3}  Theorem \ref{thm2}  is an extension of 
Theorem 3.1 in \cite{n1}, where additionally it is assumed that
\begin{equation}
\label{intfin}
\int_X \sup_{\pmb{a}\in A} \delta(x,y,\pmb{a})\mu(dy) <\infty\quad\mbox{for each}\quad x\in X.
\end{equation}
\end{remark}

\subsection{\bf Weak correlated  equilibria in  constrained discounted stochastic games} 

Let $\Psi$ be the set of all transition
probabilities from $X$ to $A,$ that is,  $\psi\in\Psi$ if 
$\psi(\cdot|x)\in \Pr(A)$ for every $x\in X $ and 
$\psi(D|\cdot)$ is $\cal F$-measurable for any Borel set $D\subset A.$ 
 A  {\it stationary correlated strategy} for the players  in the $CSG$ is a constant sequence $(\psi,\psi,\ldots),$ 
where $\psi\in \Psi.$ As in the case of stationary strategies, 
 we shall identify a correlated strategy with the element $\psi$ of this sequence.

By the Ionescu-Tulcea theorem \cite{n}, any correlated strategy 
$\psi\in \Psi$ and the initial distribution  $\eta,$ 
induce  a unique probability measure 
$\mathbb{P}_\eta^{\psi}$ on $H^\infty.$  
The expectation operator with respect to $\mathbb{P}_\eta^{\psi}$ is denoted by $\mathbb{E}_\eta^{\psi}.$ 
Then the discounted cost functionals for player $i\in {\cal N}$   are defined as
$$
J_i^\ell(\psi) = (1-\alpha)\mathbb{E}^{\psi}_\eta\left[\sum_{t=1}^\infty  \alpha^{t-1} c_i^\ell(x^t,\pmb{a}^t) \right]
$$
for all $\ell\in{\cal L}_0.$ Obviously, here at stage $t$ 
the vector of actions $\pmb{a}^t$ is chosen according to a probability measure $\psi(\cdot|x^t).$

Furthermore,  let $\psi_{-i}$ denote the projection of $\psi(\cdot|x)$ on $A_{-i}$ for every $x\in X.$
For any player $i\in{\cal N}$ and a strategy $\pi_i\in\Pi_i$ we denote by $[\psi_{-i},\pi_i]$
a multi-strategy, where player $i$ uses a strategy $\pi_i$ and the other players act as one player applying $\psi_{-i}.$
 In this case,    $J^0_i([\psi_{-i},\pi_i])$    denotes the expected discounted cost for player $i.$   
Set
$$\Delta_i(\psi_{-i})=\{\pi_i\in\Pi_i:\ J_i^\ell([\psi_{-i},\pi_i])\le \kappa_i^\ell\ \mbox{ for all } \ \ell\in{\cal L}\}.$$

\begin{definition}
\label{def3}   A strategy  
$\psi^*\in\Psi$ is called a {\it weak correlated equilibrium} in the $CSG$, if 
 for every   $i\in {\cal N} $ and $\ell \in {\cal L},$  
 $J^\ell_i (\psi^*)  \le \kappa^\ell_i$  and   for every   $i\in {\cal N}, $
 \begin{equation}
\label{J00} J^0_i(\psi^*) \le\inf_{\pi_i \in\Delta_i(\psi_{-i}^*)} J^0_i([\psi_{-i}^*,\pi_i]).
\end{equation}
\end{definition}

If  all players but $i\in {\cal N}$  accept to use 
$\psi^*$ to select an action profile in any state $x$ and player $i\in {\cal N}$ decides to play independently of
all of them by choosing a feasible strategy $\pi_i$, then the action profile for all players in ${\cal N}\setminus \{i\}$ 
is selected  with respect to the marginal probability distribution $\psi_{-i}^*(\cdot|x)$
on $A_{-i}.$  When $\psi^*$ is a weak correlated equilibrium,
then inequality (\ref{J00}) says that unilateral deviations from $\psi^*$ are not profitable. 
This is an adaptation  of the equilibrium concept, formulated by Moulin and Vial \cite{mv} 
for static games, to our dynamic game model.

In order to state our third main result, we define  $\Phi_{-i}:= \prod_{j\in {\cal N}\setminus\{i\}}\Phi_j $
and impose  the following condition.\\

\noindent{\bf  Assumption A4} \\   For each  player $i\in {\cal N},$     
$$
\sup_{\pmb{\varphi_{-i}} \in  \Phi_{-i}}  \min_ {\sigma_i\in \Phi_i}
\max_{\ell \in {\cal L}}\ (J^\ell_i([\pmb{\varphi_{-i}},\sigma_i]) -\kappa^\ell_i) <0.$$
\\

This assumption implies the standard Slater condition widely used in the literature, see \cite{a,as,dpr2,jn3}.\\   

\noindent{\bf  Assumption A5} \\   For each  player $i\in {\cal N}$ 
and any $\pmb{\varphi_{-i}}\in \Phi_{-i}$, there exists $\sigma_i\in \Phi_i$ such that   
$$J^\ell_i([\pmb{\varphi_{-i}},\sigma_i]) <\kappa^\ell_i \quad\mbox{for all}\quad  \ell \in {\cal L}. $$
\\

Assumptions {\bf A4} and {\bf A5} may seemingly be more general. Namely, we can formulate them for 
$\pi_i\in\Pi_i$  instead of $\sigma_i\in\Phi_i$ and replace the set $\Phi_i$ by $\Pi_i.$ 
However, Remark \ref{rS} implies that these formulations are in fact equivalent.

\begin{remark}
\label{remz}
From  Assumption {\bf  A4}, it follows that there exists $\zeta>0$ such that for every  player $i\in {\cal N},$     
$$ 
\sup_{\pmb{\varphi_{-i}} \in  \Phi_{-i}}  \min_ {\sigma_i\in \Phi_i}
\max_{\ell \in {\cal L}}\ (J^\ell_i([\pmb{\varphi_{-i}},\sigma_i]) -\kappa^\ell_i) <- \zeta,$$
and consequently  that for each  player $i\in {\cal N}$ 
and any $\pmb{\varphi_{-i}}\in \Phi_{-i}$, there exists $\sigma_i\in \Phi_i$ such that   
$$J^\ell_i([\pmb{\varphi_{-i}},\sigma_i]) <\kappa^\ell_i  -\zeta \quad\mbox{for all}\quad  \ell \in {\cal L}. $$
\end{remark}

\begin{theorem}
\label{thm3}  Assume {\bf A1}, {\bf A2} {\it and} {\bf A4}. Then, the $CSG$
possesses a weak correlated equilibrium.
\end{theorem}

The proof is given in Section  \ref{s4}. 

\begin{remark}
\label{r4}  The existence of a weak correlated equilibrium in an unconstrained case was proved by 
Nowak  \cite{n2}  under  additional integrability condition (\ref{intfin}).
\end{remark}

\begin{remark}
\label{r5a} 
If $\psi^*$ is a stationary weak correlated  equilibrium in an  $ARAT$ game,
then $(\psi_{1},\psi_{2},...,\psi_N)$ is a stationary Nash equilibrium in the $ARAT$ game. 
Thus, Theorem \ref{thm3} implies the main result of Dufour and Prieto-Rumeau \cite{dpr2}, 
if  the action sets are independent of the state. 
However, their proof is more direct in the sense that it is not based on an approximation 
by games with discrete state spaces. Instead, they directly apply a fixed point theorem. 
An extension to the case of action spaces depending 
on the state variable raises some additional technical issues.
\end{remark}

\section{Approximating games with countable state spaces and proofs of Theorems
 \ref{thm1}  and \ref{thm2} } \label{s3}

In this section, we define a class of games that resemble 
stochastic games with a countable state space. Using them we can approximate the original 
game and apply the results on existence of stationary equilibria in discounted games with 
countably many states proved by Federgruen \cite{f} (unconstrained case)  
and Ja{\'s}kiewicz and Nowak \cite{jn3} (constrained case). 

Let ${\cal C}(A)$ be the Banach space of all real-valued continuous 
functions on $A$ endowed with the maximum norm $\|\cdot\|.$ Let ${\cal C}_b
= \{w_1,w_2,...\}$  denote the countable dense subset in the ball 
$\{w\in {\cal C}(A): \|w\|\le b\}$ in ${\cal C}(A),$ where $b\ge |c^\ell_i(x,\pmb{a})|$
for all $i\in {\cal N}$ $\ell \in {\cal L}_0,$ $(x,\pmb{a})\in \mathbb{K}.$ 

We write ${\cal L}^1$ to denote  the Banach space  ${\cal L}^1(X,{\cal F},\mu)$ 
of all absolutely integrable real-valued measurable functions on $X$ with the norm
$$\| v \|_1= \int_X|v (y)|\mu(dy), \quad v \in {\cal L}^1.$$
Let ${\cal C}(A,{\cal L}^1)$ be the space of all ${\cal L}^1$-valued 
continuous functions on $A$ with the norm
$$\|\lambda\|_c = \max_{\pmb{a}\in A}\int_X|\lambda(y,\pmb{a})|\mu(dy).$$
Here an element of ${\cal C}(A,{\cal L}^1)$ is written as a product measurable
function $\lambda: X\times A \to \mathbb{R}$ such that
$\lambda(\cdot,\pmb{a}) \in {\cal L}^1$ for each $\pmb{a}\in A$ and   
$$\|\lambda(\cdot,\pmb{a}^n) - \lambda(\cdot,\pmb{a})\|_1 = 
\int_X|\lambda(y,\pmb{a}^n) - \lambda(y,\pmb{a})|\mu(dy) \to 0 
\quad \mbox{as}\quad \pmb{a}^n \to \pmb{a}, \  n\to\infty.$$
By Lemma 3.99 in \cite{ab}, the space  ${\cal C}(A,{\cal L}^1)$ is separable.
Assumption {\bf A2} implies that 
${\cal D}:= \{ \delta(x,\cdot,\cdot): x \in X \} \subset {\cal C}(A,{\cal L}^1)$
is also a separable space when endowed with the relative topology.
Therefore, there exists a subset $\{x_k: k \in \mathbb{N}\}$  of the state space $X$ such that 
the set  $\{\delta(x_k,\cdot,\cdot): k \in \mathbb{N} \}$ is dense in ${\cal D}.$

For any player $i\in {\cal N},$ and positive integers
$m_{i\ell},$  $\ell\in {\cal L}_0,$ we put $\overline{m}_i =(m_{i0},m_{i1},...,m_{iL}).$  
Then, given any   $\gamma >0 $, we define
 $B^\gamma(i,\overline{m}_i)$  as the set of all states $x\in X$ such that 
\begin{equation}
\label{bi}
\sum_{\ell=0}^L \|c_i^\ell(x,\cdot)-  w_{m_{i\ell}}\|<  \gamma.
\end{equation} 
For any $k \in \mathbb{N},$ let 
\begin{equation}
\label{bk}
B_k^\gamma := \{x\in X: \|\delta(x,\cdot,\cdot)-\delta(x_k,\cdot,\cdot) \|_c = 
\max_{\pmb{a}\in A}\int_X|\delta(x,y,\pmb{a})-\delta(x_k,y,\pmb{a})|\mu(dy) <\gamma \}. 
\end{equation}
 It is obvious that the sets $B_k^\gamma$ and    $B^\gamma(i,\overline{m}_i)$ 
 belong to $\cal F$ and the union of all sets
$$B_k^\gamma\cap B^\gamma(1,\overline{m}_1) \cap \ldots\cap B^\gamma(N,\overline{m}_N)$$  
is the whole state space $X.$  
Indeed, if $x\in X,$ then there exists $k\in\mathbb{N}$ such that $x\in B_k^\gamma$ and, 
for any player $i\in {\cal N},$ 
there exist functions $w_{m_{i\ell}}\in {\cal C}_b, $ and thus $\overline{m}_i $ 
such that (\ref{bi}) holds.

Let $\xi$ be a fixed one-to-one correspondence between the sets $\mathbb{N}$ 
and $\mathbb{N}\times \mathbb{N}^{N(L+1)}.$
Assuming that $j\in \mathbb{N}$ and    
$\xi(j)=  (k,\overline{m}_1,...,\overline{m}_N),$ we put 
$$Y_j^\gamma := 
B_k^\gamma\cap B^\gamma(1,\overline{m}_1) \cap \ldots\cap 
B^\gamma(N,\overline{m}_N).$$
We can assume without loss of generality that $Y^\gamma_1 \not= \emptyset.$  
Next, we set $X^\gamma_1= Y_1^\gamma$  and
$$X_\tau^\gamma = Y_\tau^\gamma - \bigcup_{t<\tau} X^\gamma_t, \quad\mbox{for}\quad 
\tau\in\mathbb{N}\setminus \{1\}.$$
Omitting empty sets $X_\tau^\gamma$ we obtain a subset $\mathbb{N}_0 \subset \mathbb{N}$ such that
$${\cal P}^\gamma = \{ X_j^\gamma: j\in \mathbb{N}_0 \}$$
is a measurable partition of the state space $X.$ Choose any  $n \in \mathbb{N}_0.$  Then, $\xi(n)$
is a unique sequence in ${\mathbb{N}}\times\mathbb{N}^{N(L+1)}$ that depends on $n$ 
and, therefore, we can write 
$\xi(n)= (k^n,\overline{m}_1^n,...,\overline{m}_N^n)$ where
$\overline{m}_i^n = (m_{i0}^n, m_{i1}^n,...,m_{iL}^n),$
$i\in {\cal N}.$   Next, for each $x\in X_n^\gamma,$ we define
\begin{equation}
\label{deltaci}
\delta^\gamma(x, y,\pmb{a}) := \delta(x_{k^n},y,\pmb{a}) \quad\mbox{for }y\in X\quad\mbox{and}\quad
c^{\ell,\gamma}_i (x,\pmb{a}):= 
w_{m_{i,\ell}^n}(\pmb{a}) \quad  \mbox{ for all }\ell \in {\cal L}_0,\  i\in {\cal N}. 
\end{equation}
From (\ref{bi}), (\ref{bk}) and  (\ref{deltaci}), it follows that
for each $n\in \mathbb{N}_0$ and $x\in X_n^\gamma,$ we have  
\begin{equation}
\label{eps1}
\| c_i^\ell(x,\cdot)-  c^{\ell,\gamma}_i (x, \cdot)\| <\gamma\quad  \mbox{ for all } \ell\in {\cal L}_0
\end{equation}
and 
\begin{equation}
\label{eps2}
\|\delta(x,\cdot,\cdot)-
\delta^\gamma(x,\cdot,\cdot)\|_c=
\max_{\pmb{a}\in A} \int_X|\delta(x, y,\pmb{a})-
\delta^\gamma(x,  y,\pmb{a})|\mu(dy) <\gamma.
\end{equation}

The original game defined in Section \ref{s2} is now denoted by $\cal G$. We use ${\cal G}^\gamma$ 
to denote the game, where the cost functions are $ c^{\ell,\gamma}_i,$ $\ell\in{\cal L}_0$ and $i\in{\cal N}$, 
 and the transition probability is
$$
p^\gamma(B|x,\pmb{a}) = \int_X
\delta^\gamma(x,  y,\pmb{a})\mu(dy), \quad  B \in {\cal F}.
$$
Note that $ c^{\ell,\gamma}_i(x,\pmb{a})$ and $ p^\gamma(B|x,\pmb{a})$ 
are constant functions of $x$ on every set $X_n^\gamma.$

The discounted expected costs in the game ${\cal G}^\gamma$ under 
a multi-strategy $\pmb{\pi}\in \Pi$ are denoted by
$$ J_i^{\ell,\gamma}(\pmb{\pi})(x)\quad\mbox{and}\quad
 J_i^{\ell,\gamma}(\pmb{\pi}) = \int_X J_i^{\ell,\gamma}(\pmb{\pi})(x)\eta(dx).$$

Let 
\begin{equation}
\label{eps}
\epsilon (\gamma) := \frac{\gamma (1-\alpha+b\alpha)}{1-\alpha}.
\end{equation}
From  (\ref{eps1}),  (\ref{eps2})  and Lemma 4.4 in \cite{n1}, we conclude the following auxiliary result.\\

\begin{lemma}
\label{l1}  For each $i\in {\cal N}$ and $\ell\in {\cal L}_0,$ we have 
$$\sup_{x\in X} \sup_{\pmb{\pi} \in \Pi}|J_i^\ell(\pmb{\pi})(x)-  J_i^{\ell,\gamma}(\pmb{\pi})(x)|\le \epsilon(\gamma).
$$
\end{lemma}

With ${\cal G}^\gamma$ we associate a stochastic game ${\cal G}^\gamma_c$ 
with the countable state space $\mathbb{N}_0  \subset \mathbb{N}$, the costs given by 
\begin{equation}
\label{costcssp}
\widehat{c}^{\ell,\gamma}_i( n,\pmb{a}):=c^{\ell,\gamma}_i( x,\pmb{a}), 
\quad x\in X^\gamma_n,\quad  n \in \mathbb{N}_0, \quad \pmb{a}\in A,
\end{equation}
and transitions defined as 
\begin{equation}
\label{trcssp}
\widehat{p}^\gamma(\tau|n,\pmb{a}):=\delta^\gamma(X_\tau^\gamma|x,\pmb{a}), 
\quad x\in X^\gamma_n,\quad   n,  \tau\in \mathbb{N}_0, \quad \pmb{a}\in A. 
\end{equation}
Note that the right-hand sides in (\ref{costcssp}) and (\ref{trcssp}) 
are independent of $x$ in $X_n^\gamma$  and thus the costs and transitions
above are well-defined.  A stationary strategy for player $i\in {\cal N}$ 
in the game ${\cal G}^\gamma_c$  is a transition probability $f_i$  from $\mathbb{N}_0$ to $A_i.$ 
The set of all stationary strategies for player $i\in {\cal N}$ 
in this game is denoted by $F_i.$ We  put $F:= \prod_{i\in {\cal N}}F_i.$
 
The expected discounted costs in the game
${\cal G}^\gamma_c$   under stationary multi-strategy $\pmb{\pi}$ 
are denoted by
$$\widehat{J}^{\ell,\gamma}_i(\pmb{\pi})(n), \ n\in\mathbb{N}_0,\quad\mbox{and}\quad
\widehat{J}^{\ell,\gamma}_i(\pmb{\pi}) =
\sum_{n\in \mathbb{N}_0}\widehat{J}^{\ell,\gamma}_i(\pmb{\pi})(n)\eta(X_n^\gamma).
$$

Let $ \Phi^\gamma_i$ be the set of all piecewise constant  stationary strategies 
of player $i\in {\cal N} $ in the game ${\cal G}^\gamma.$  A strategy $\varphi_i\in \Phi^\gamma_i$, if, 
for each $n\in\mathbb{N}_0,$ there exists a probability measure $\nu_n$ on $A_i$ 
such that  $\varphi_i (da_i|x)= \nu_n(da_i)$ for all  $x\in X_n^\gamma.$   We put
$\Phi^\gamma = \prod_{i\in {\cal N}} \Phi^\gamma_i.$ 

Let $\pmb{f}= (f_1,...,f_N) \in F$ and $\pmb{\varphi}= (\varphi_1,...,\varphi_N)\in \Phi^\gamma $ be such that
\begin{equation}
\label{R0}
\varphi_i(da_i|x)= f_i(da_i|n)\quad\mbox{for all}\quad i\in {\cal N},\ n\in\mathbb{N}_0,\ x\in X_n^\gamma.
\end{equation}
Then, for each $i\in {\cal N}$, $\ell \in {\cal L}_0,$ $n\in\mathbb{N}$ and $x\in X_n^\gamma,$
\begin{equation}
\label{R1} 
J^{\ell,\gamma}_i(\pmb{\varphi})(x)= \widehat{J}^{\ell,\gamma}_i(\pmb{f})(n)  
\end{equation} 
and 
\begin{equation}
\label{R2} 
J^{\ell,\gamma}_i(\pmb{\varphi}) = \widehat{J}^{\ell,\gamma}_i(\pmb{f}). 
\end{equation}
Equations (\ref{R1}) and (\ref{R2}) show that ${\cal G}^\gamma$ with the strategy sets
$\Phi^\gamma_i$ can be recognised as a game with a countable state space. This observation plays an 
important role in the proof, because we can apply a result for games on countable state spaces.

\begin{proof}{\it of Theorem \ref{thm1}.} Let $\varepsilon >0$ and $i \in{\cal N}$.
Choose $\gamma>0$ in (\ref{eps}) such that $\epsilon(\gamma) < \varepsilon/2.$  
By  Assumption {\bf A3} and Remark \ref{rS} we imply that
for any multi-strategy $\pmb{\varphi}\in \Phi^\gamma$  
there exists $\sigma_i\in\Phi_i$ such that
\begin{equation}\label{LA0}
 J^\ell_i([\pmb{\varphi_{-i}},\sigma_i]) \le \kappa^\ell_i  \quad \mbox{for all}\ \ell \in {\cal L}.
\end{equation}
By Lemma \ref{A1} in Appendix,   there exists a piecewise constant Markov strategy $\overline{\pi}_i$ such that
$$
J_i^{\ell,\gamma}([\pmb{\varphi_{-i}},\sigma_i])=J_i^{\ell,\gamma}([\pmb{\varphi_{-i}},\overline{\pi}_i])
$$ 
for all $\ell\in{\cal L}_0.$ 
By Lemma \ref{l1} and (\ref{LA0}) we conclude that
$$
 J_i^{\ell,\gamma}([\pmb{\varphi_{-i}},\overline{\pi}_i])<\kappa^\ell_i +\frac\varepsilon 2 \quad \mbox{for all}\ \ell \in {\cal L}.
$$
This means that the approximating game ${\cal G}^\gamma$ satisfies the Slater condition with the constants 
$\kappa^\ell_i +\frac\varepsilon 2,$ $\ell\in{\cal L}.$  Note that the constraint constants in ${\cal G}^\gamma$
are also equal $\kappa^\ell_i +\frac\varepsilon 2,$ $\ell\in{\cal L}.$   Therefore, the associated game ${\cal G}_c^\gamma$
also satisfies the Slater condition with the same constants $ \kappa^\ell_i +\frac\varepsilon 2,$ $\ell\in{\cal L}.$
Making use of  Corollary 2 in \cite{jn3},  we infer that the game ${\cal G}_c^\gamma$ 
possesses a stationary Nash equilibrium 
$\pmb{f}^*= (f_1^*,...,f_N^*).$ 
Define $\pmb{\varphi}^* =
 (\varphi_1^*,...,\varphi_N^*)\in  \Phi^\gamma$ as in (\ref{R0})  with $ \pmb{\varphi}=\pmb{\varphi}^*$  and $f=f^*.$ 
Then,
$$
J_i^{0,\gamma}(\pmb{\varphi}^*)\le J_i^{0,\gamma}([\pmb{\varphi_{-i}^*},\hat{\pi}_i])$$
for any piecewise constant strategy $\hat{\pi}_i$ such that
$$ J_i^{\ell,\gamma}([\pmb{\varphi_{-i}^*},\hat{\pi}_i])\le\kappa^\ell_i +\frac\varepsilon 2\quad \mbox{for all}\ \ell \in {\cal L}.$$
We now show that $\pmb{\varphi }^*$ is an 
$\varepsilon$-equilibrium in the original game.
Note that for every player $i\in{\cal N}$
$$
J_i^{\ell,\gamma}( \pmb{\varphi^*} )=
J_i^{\ell,\gamma}([\pmb{\varphi_{-i}^*},\varphi_i^*])\le\kappa^\ell_i +\frac\varepsilon 2\quad \mbox{for all}\ \ell \in {\cal L}.
$$
Hence, for every player $i\in{\cal N}$
$$
J_i^{\ell }( \pmb{\varphi^*} ) 
 \le\kappa^\ell_i + \varepsilon \quad \mbox{for all}\ \ell \in {\cal L},
$$
i.e.,  condition (\ref{def1e1}) holds. 
Consider  any feasible strategy $\pi_i\in\Delta_i(\pmb{\varphi_{-i}^*}),$ i.e.,
\begin{equation}\label{LA2}
 J^\ell_i([\pmb{\varphi_{-i}^*},\pi_i]) \le \kappa^\ell_i,  \quad \mbox{for all}\quad \ell \in {\cal L}.
 \end{equation}
 Applying Remark \ref{rS}, we deduce that there exists a strategy 
$\sigma_i\in\Phi_i$ such that
\begin{equation}\label{LA3}
J^\ell_i([\pmb{\varphi_{-i}^*},\pi_i])=J^\ell_i([\pmb{\varphi_{-i}^*},\sigma_i])\quad\mbox{for all}\quad \ell\in{\cal L}_0.
\end{equation}
Then, by Lemma \ref{A1} in Appendix,  there exists a piecewise constant Markov strategy 
$\overline{\pi}_i$ such that
\begin{equation}\label{LA}
J_i^{\ell,\gamma}([\pmb{\varphi_{-i}^*},\sigma_i])=J_i^{\ell,\gamma}([\pmb{\varphi_{-i}^*},\overline{\pi}_i])\quad
 \mbox{for all $\ell\in{\cal L}_0.$} \end{equation}
Moreover, by (\ref{LA}), Lemma \ref{l1}, (\ref{LA3}) and (\ref{LA2}), 
for every $\ell\in{\cal L},$ we have  
$$
J_i^{\ell,\gamma}([\pmb{\varphi_{-i}^*},\overline{\pi}_i])\le  
J^\ell_i([\pmb{\varphi_{-i}^*},\sigma_i])+\frac\varepsilon 2 
= J^\ell_i([\pmb{\varphi_{-i}^*},\pi_i])+\frac\varepsilon 2
\le \kappa^\ell_i
+\frac\varepsilon 2.
$$
 In other words,
$\overline{\pi}_i$ is a feasible strategy in ${\cal G}^\gamma$.
Therefore,  by Lemma \ref{l1}, (\ref{LA}) and (\ref{LA3}), we infer
\begin{eqnarray*}
J_i^{0}(\pmb{\varphi}^*)&\le& J_i^{0,\gamma}(\pmb{\varphi}^*)+\frac\varepsilon2\le 
J_i^{0,\gamma}([\pmb{\varphi_{-i}^*},\overline{\pi}_i])+\frac\varepsilon2=
J_i^{0,\gamma}([\pmb{\varphi_{-i}^*},\sigma_i])+\frac\varepsilon2\\
&<& J_i^{0}([\pmb{\varphi_{-i}^*},\sigma_i]) +\varepsilon=
J_i^{0}([\pmb{\varphi_{-i}^*},\pi_i])+\varepsilon.
\end{eqnarray*}
This fact together with  (\ref{LA2}) implies that (\ref{def1e2}) holds.  \hfill 
$\Box$\end{proof}

\begin{proof}{\it of Theorem \ref{thm2}.} 
Let $\varepsilon >0$ be fixed. Choose $\gamma>0$ in (\ref{eps}) such that $\epsilon(\gamma) < \varepsilon/2.$ 
By Theorem \ref{thm1} in \cite{f}, the game
${\cal G}^\gamma_c$ has a stationary equilibrium $\pmb{f}^*= (f_1^*,...,f_N^*).$ 
Define $\pmb{\varphi}^* = (\varphi_1^*,...,\varphi_N^*)\in  \Phi^\gamma$ as in the proof of Theorem \ref{thm1}. 
Then we have
\begin{equation}
\label{R3}
J_i^{0,\gamma}(\pmb{\varphi}^*)(x) = \inf_{\phi_i\in  \Phi_i^\gamma}
J_i^{0,\gamma}([\pmb{\varphi_{-i}^*},\phi_i])(x), \quad i\in {\cal N},\ x\in X. 
\end{equation} 
As in Lemma 4.1 in \cite{n1}, we can prove that
\begin{equation}
\label{R4}
\inf_{\phi_i\in \Phi^\gamma_i}
J_i^{0,\gamma}([\pmb{\varphi_{-i}^*},\phi_i])(x)=
\inf_{\phi_i\in  \Phi_i}
J_i^{0,\gamma}([\pmb{\varphi_{-i}^*},\phi_i])(x), \quad i\in {\cal N},\ x\in X. 
\end{equation}
By (\ref{R3}) and (\ref{R4}), we get
$$
J_i^{0,\gamma}(\pmb{\varphi}^*)(x) = \inf_{\phi_i\in  \Phi_i}
J_i^{0,\gamma}([\pmb{\varphi_{-i}^*},\phi_i])(x), \quad i\in {\cal N},\ x\in X. 
$$
This equality and Lemma \ref{l1} imply that
\begin{equation}
\label{R5}
J_i^{0 }(\pmb{\varphi}^*)(x)-\varepsilon \le    \inf_{\phi_i\in   \Phi_i}
J_i^{0 }([\pmb{\varphi_{-i}^*},\phi_i])(x), \quad i\in {\cal N},\ x\in X. 
\end{equation} 
By standard methods in discounted dynamic programming
\cite{bs,n1}, we have
$$
\inf_{\phi_i\in  \Phi_i}
J_i^{0}([\pmb{\varphi_{-i}^*},\phi_i])(x)=
\inf_{\sigma_i\in  \Pi_i}
J_i^{0}([\pmb{\varphi_{-i}^*},\sigma_i])(x), \quad i\in {\cal N},\ x\in X. 
$$
This fact and (\ref{R5}) imply that
$$
J_i^{0 }(\pmb{\varphi}^*)(x)-\varepsilon \le 
\inf_{\sigma_i\in  \Pi_i}
J_i^{0}([\pmb{\varphi_{-i}^*},\sigma_i])(x), \quad i\in {\cal N},\ x\in X,
$$
which completes the proof. \hfill $\Box$ \end{proof}

\begin{remark}
\label{r6}  The proof of Theorem \ref{thm2} 
 is similar to that of Theorem 3.1 in \cite{n1}, 
but it has one important change implying
that the restrictive condition (\ref{intfin}) can be dropped.
 \end{remark}

\section{Young measures and the proof of Theorem \ref{thm3}}
\label{s4}

Let $\vartheta:=(\eta+\mu)/2.$
A function $c:\mathbb{K}\to\mathbb{R}$ is Carath{\'e}odory, if it is 
product measurable on $\mathbb{K}$, 
$c(x,\cdot)$ is continuous on $A$ for each $x\in X$  and 
$$\int_X\max\limits_{\pmb{a}\in A}|c(x,\pmb{a})|\vartheta(dx)<\infty.$$
Let $\Psi^\vartheta$ be the space of all $\vartheta$-equivalence classes of 
functions in $\Psi.$ The elements of $\Psi^\vartheta$ are called Young measures.
Note that the expected discounted cost functionals are well-defined for all elements of $\Psi^\vartheta.$  
More precisely, if $\psi^\vartheta \in \Psi^\vartheta,$ then
$J^\ell_i(\psi)$ is the same for all representatives $\psi$ of $\psi^\vartheta$ in $\Psi$ 
and we can understand $J^\ell_i(\psi^\vartheta)$ as $J^\ell_i(\psi).$
We shall identify in notation $\psi^\vartheta$ with its representative $\psi$ 
and omit the superscript $\vartheta.$ 

We assume that the space $\Psi^\vartheta$ is endowed with the weak* topology.  
Since ${\cal F}$ is countably generated, $\Psi^\vartheta$ is metrisable. 
Moreover, since the set $A$ is compact, 
$\Psi^\vartheta$ is a {\it compact convex} subset of a locally convex linear topological space. 
For a detailed discussion of these issues consult with \cite{bal} or Chapter 3 in \cite{fg}. 
Here, we  recall that
$\psi^n \to^* \psi^0$ in $\Psi^\vartheta$  as $n\to\infty$  if and only if for every
Carath{\'e}odory function $c:\mathbb{K}\to \mathbb{R}$, we have
$$ \lim\limits_{n\to\infty}\int_X\int_{A}c(x,\pmb{a})\psi^n(d\pmb{a}|x)\vartheta(dx)=
\int_X\int_{A}c(x,\pmb{a})\psi^0(d\pmb{a}|x)\vartheta(dx).$$\\

We now choose   $\varepsilon_n>0$ such that $\varepsilon_n \searrow 0$ as $n\to\infty $ and define  
\begin{equation}
\label{gn}
\gamma_n:=\frac{\varepsilon_n(1-\alpha)}{(1-\alpha+b\alpha)}.
\end{equation}
In other words, $\epsilon(\gamma_n)=\varepsilon_n$ or $\gamma_n=\epsilon^{-1}(\varepsilon_n).$
From Theorem \ref{thm1},  it follows that there exists a profile of stationary piecewise constant strategies 
$$\pmb{\psi^n}=(\psi_{1}^n,\ldots,\psi_{N}^n)\in \Phi^{\gamma_n},$$ 
which comprises  an approximate equilibrium in the $CSG$
for $\varepsilon_n$ and at the same time an equilibrium in the
corresponding constrained game ${\cal G}^{\gamma_n}$ with  $\gamma_n$ as in (\ref{gn}) 
and the constraint constants $\kappa^\ell_i+\frac{\varepsilon_n}2.$ 
 
Define the product measure on $A,$ for every $x\in X$ and $n\in \mathbb{N}$  as 
\begin{equation}\label{gi}
\psi^n(\cdot|x):=\psi_{1}^n(\cdot|x)\otimes\ldots\otimes\psi_{N}^n(\cdot|x).
\end{equation}
We use   $\psi^n$  to denote the class in $\Psi^\vartheta$ 
whose representative  is this transition probability. 
Without loss of generality, we may 
assume that $\psi^n$ converges in the weak* topology to some 
$\psi^* \in  \Psi^\vartheta$  as $n\to\infty.$   

We shall need the following results. The first one  is a consequence of Lemma \ref{l1} and the fact that 
$J_i^{\ell,\gamma_n}(\pmb{\psi^n})=  J_i^{\ell,\gamma_n}(\psi^n)$ and $J_i^{\ell}(\pmb{\psi^n})=  J_i^{\ell}(\psi^n).$

\begin{lemma}
\label{l1a}  For each $i\in {\cal N}$ and $\ell\in {\cal L}_0,$ we have 
$$
\sup_{\psi  \in \Psi}|J_i^\ell(\psi)- 
 J_i^{\ell,\gamma_n}(\psi)|\le \varepsilon_n, $$
$$
\sup_{\psi_{-i} \in \Psi_{-i}}\sup_{\pi_i\in\Pi_i}|J_i^\ell([\psi_{-i},\pi_i])- 
 J_i^{\ell,\gamma_n}([\psi_{-i},\pi_i])|\le \varepsilon_n, 
$$
where $\gamma_n$ is as in (\ref{gn}).
\end{lemma}

\begin{lemma}
\label{l2}   If     $n\to\infty,$   then for any $\ell\in {\cal L}_0$\\
(a) $J_i^{\ell,\gamma_n}(\psi^n) \to J_i^{\ell}(\psi^*),$  \\
(b) $J_i^{\ell,\gamma_n}([\psi_{-i}^n,\phi_i]) \to J_i^{\ell}([\psi_{-i}^*,\phi_i])$
for any $\phi_i\in\Phi_i.$
\end{lemma}

\begin{proof} For part (a) we first use  the triangle inequality
$$|J_i^{\ell,\gamma_n}(\psi^n) - J_i^{\ell}(\psi^*)|\le |J_i^{\ell,\gamma_n}(\psi^n) - 
J_i^{\ell}(\psi^n)|+ |J_i^{\ell}(\psi^n) - J_i^{\ell}(\psi^*)|.$$
The first term on the right-hand side converges to $0$ by Lemma \ref{l1a} and the definition of $\psi^n,$ 
whereas the convergence to $0$ of the
the second term follows  from Lemma 4.1 in \cite{jn2} and the fact that $|J_i^{\ell}(\cdot)|\le b$ 
for every $i\in{\cal N}$ and $\ell\in{\cal L}_0.$
Part (b) is proved as point (a) by using the Fubini theorem and noting that the elements in 
$\Psi^\vartheta$ induced by $\psi_{-i}^n$ in (\ref{gi})
and $\phi_i$  converge  in the weak* sense to the element of 
$\Psi^\vartheta$ induced by  $\psi_{-i}^*$ and $\phi_i.$   \hfill $\Box$
\end{proof}

Let $i\in{\cal N}.$   Consider a Markov decision process
with player $i$ as a decision maker and the transition probability $$
q^{\gamma_n}(dy|x,a_i)=   \int_{A_{-i}}p^{ \gamma_n}(dy|x,[\pmb{a_{-i}},a_i])
\psi_{-i}^n(d\pmb{a_{-i}}|x), \quad (x,a_i)\in\mathbb{K}_i.
$$
 Let $1_D$ be the indicator of the set $D\subset  X\times A.$
The associated occupation measure, when player $i$ uses a stationary strategy 
$\varphi_i\in\Phi_i$   is defined as follows
\begin{equation}\label{om}
 \theta^{\gamma_n}_{\varphi_i}(B\times C)=
(1-\alpha)\sum_{t=1}^\infty\alpha^{t-1} {\cal E}_\eta^{\varphi_i} 1_{B\times C}(x^t,a_i^t)
\end{equation}
for any $B\in{\cal F}$ and a Borel set $C$ in $A_i.$ We use the symbol
${\cal E}^{\varphi_i}_\eta$ to denote the expectation operator corresponding to the unique probability measure
 induced by 
$\varphi_i\in\Phi_i$, the initial distribution $\eta$ and the transition probability $q^{\gamma_n}.$
For $\ell \in {\cal L}_0,$ $x\in X$ and $a_i\in A_i,$ set 
$$
c_i^{\ell,  \gamma_n }(x,a_i):= \int_{A_{-i}}c_i^{\ell,\gamma_n}(x,[\pmb{a_{-i}},a_i])
\psi_{-i}^n(d\pmb{a_{-i}}|x).$$

\begin{proof} {\it of Theorem \ref{thm3}.} Observe that Assumption {\bf A4} implies {\bf A3}. 
We consider the weak* limit $\psi^* 
\in \Psi^\vartheta$ mentioned above and denote its representative in $\Psi$ by the same letter. 

We shall show that   $\psi^*$   is a weak correlated equilibrium.
By Theorem \ref{thm1}, $J^\ell_i(\pmb{\psi^n}) =J^\ell_i(\psi^n) \le \kappa^\ell_i +\varepsilon_n$ for all
$i\in {\cal N}$ and $\ell \in {\cal L}.$  Using Lemma \ref{l2}(a), we conclude that
$$J^\ell_i(\psi^*)=\lim_{n\to\infty}J^\ell_i(\psi^n) \le \kappa^\ell_i,\quad i\in {\cal N},\ \ell \in {\cal L},$$
i.e., $\psi^*$ is feasible. 

Take (if possible)  any feasible strategy in the $CSG$ for player $i\in {\cal N}$, i.e.,  $ \pi_i\in\Pi_i$ such that
$$J_i^{\ell} ([\psi_{-i}^*,\pi_i])\le \kappa_i^\ell\quad\mbox{for all}\quad \ell\in{\cal L}.$$
By Remark  \ref{rS}  that there exists a strategy 
$\phi_i\in\Phi_i$ such that
$$ J^\ell_i([\psi_{-i}^*,\pi_i])=J^\ell_i([\psi_{-i}^*,\phi_i])\quad\mbox{for all } \ell\in{\cal L}_0.$$

$1^\circ$ Assume first that 
\begin{equation}
\label{case1}
J^\ell_i([\psi_{-i}^*,\pi_i])=J_i^{\ell} ([\psi_{-i}^*,\phi_i])< 
\kappa_i^\ell\quad\mbox{for all}\quad \ell\in{\cal L}.
\end{equation}
From this inequality and Lemma \ref{l2}(b), we infer that there exists $N_1\in\mathbb{N}$ such that
$$J_i^{\ell,\gamma_n} ([\psi_{-i}^n,\phi_i])< \kappa_i^\ell\quad\mbox{for all}
\quad \ell\in{\cal L}\quad\mbox{and}\quad n\ge N_1.$$
For every $n\ge N_1$  and Lemma \ref{A1} 
in Appendix we conclude  the existence 
of a piecewise constant Markov strategy $\overline{\pi}_i$ (that may depend on $n$)
such that
$$J_i^{\ell,\gamma_n} ([\psi_{-i}^n,\phi_i])=J_i^{\ell,\gamma_n} ([\psi_{-i}^n,\overline{\pi}_i])
\quad\mbox{for all}\quad \ell\in{\cal L}_0.$$
Hence, it must hold 
$$J_i^{0,\gamma_n}(\psi^n)\le J_i^{0,\gamma_n}([\psi^n_{-i},\overline{\pi}_i])=
J_i^{0,\gamma_n} ([\psi^n_{-i},\phi_i]).$$
In other words, for every $n\ge N_1$ we have 
$$J_i^{0,\gamma_n}(\psi^n)\le J_i^{0,\gamma_n} ([\psi_{-i}^n,\phi_i]).$$
Letting $n\to\infty$ and making use of Lemma \ref{l2}, we infer 
$$
J_i^{0}(\psi^*)\le J_i^{0} ([\psi_{-i}^*,\phi_i])=J^0_i([\psi_{-i}^*,\pi_i])
$$
for any feasible strategy $\pi_i\in\Pi_i $ such that (\ref{case1}) holds. \\  

$2^\circ$ Assume now that there is   player $i\in {\cal N}$ and an index $\ell_0\in{\cal L}$ such that 
\begin{equation}
\label{case2}
J^{\ell_0}_i([\psi_{-i}^*,\pi_i])=J_i^{\ell_0} ([\psi_{-i}^*,\phi_i])= \kappa_i^{\ell_0}.
\end{equation}
From  the proof of  Lemma \ref{l2}(b) it follows that  there exists a sequence $e_n\ \to 0$ as $n\to\infty,$ $e_n>0,$ such that 
$$
J_i^{\ell}([\psi_{-i}^n,\phi_i])\le J_i^{\ell} ([\psi_{-i}^*,\phi_i])+e_n\le \kappa_i^\ell +e_n
\qquad\mbox{ for all}\quad \ell\in{\cal L}.
$$
By Remark \ref{remz},  we can find $\zeta>0$ such that for every $n\in\mathbb{N}$ 
there exists a strategy $\sigma^n_i\in\Phi_i$  
such that 
$$
J_i^{\ell} ([\psi_{-i}^n,\sigma^n_i])< \kappa_i^\ell -\zeta \quad\mbox{for all}\quad \ell\in{\cal L}.
$$
Hence, by Lemma \ref{l1a}, we conclude
$$
J_i^{\ell,\gamma_n} ([\psi_{-i}^n,\phi_i])-\varepsilon_n\le J_i^{\ell} 
([\psi_{-i}^n,\phi_i])\le \kappa_i^\ell+e_n
 \quad\mbox{for all}\quad \ell\in{\cal L}
$$
and 
$$
J_i^{\ell,\gamma_n} ([\psi_{-i}^n,\sigma^n_i])-\varepsilon_n\le 
J_i^{\ell} ([\psi_{-i}^n,\sigma^n_i])< \kappa_i^\ell-\zeta
 \quad\mbox{for all}\quad \ell\in{\cal L}.
$$
Let $N_2\in\mathbb{N}$ be such that
$\varepsilon_{N_2}<\zeta.$ For $n\ge N_2$ set
$$\xi_n:=\frac{\varepsilon_n+e_n}{ \zeta+e_n}$$
and observe that $\xi_n\to 0$ as $n\to\infty$ and $\xi_n\in(0,1).$
Let $\theta_{\phi_i}^{\gamma_n}$ and $ \theta_{\sigma^n_i}^{\gamma_n}$ be two occupation measures 
defined as in (\ref{om}). By Proposition 3.9 in \cite{dpr2},
we  define a sequence of  occupation  measures  as follows
$$ \theta^n:=\xi_n\theta_{\sigma^n_i}^{\gamma_n}+(1-\xi_n)\theta_{\phi_i}^{\gamma_n}.$$
Then, for  all $\ell\in{\cal L}_0$ it holds
\begin{equation}\label{aa3}
\int_{X\times A_i} c^{\ell,\gamma_n}_i(x,a_i)\theta^n(dx\times da_i)
= \xi_nJ_i^{\ell,\gamma_n} ([\psi_{-i}^n,\sigma^n_i]) +
(1-\xi_{n}) J_i^{\ell,\gamma_n} ([\psi_{-i}^n,\phi_i]).\end{equation}
Hence,  for $n\ge N_2$ and all $\ell\in{\cal L}$, from (\ref{aa3}), we have
\begin{eqnarray}\label{aa1} \nonumber
\int_{X\times A_i} c^{\ell,\gamma_n}_i(x,a_i)\theta^n(dx\times da_i)
&\le&  \xi_n (\kappa_i^\ell+\varepsilon_n -\zeta)+(1-\xi_n)( \kappa_i^\ell +\varepsilon_n+e_n)\\&=&
-\xi_n(e_n+\zeta)+\kappa_i^\ell+\varepsilon_n +e_n \le\kappa_i^\ell<\kappa_i^\ell+\frac{\varepsilon_n}2.
\end{eqnarray}
By Lemma 2.3 in \cite{dpr2} or Theorem 3.2 in \cite{fg} for every $n\ge N_2,$ 
there exists a stationary strategy $\chi^n_i\in \Phi_i$ such that
$\theta^n$ can be written  as in (\ref{om}) with ${\cal E}_\eta^{\varphi_i} $ 
replaced by  ${\cal E}_\eta^{\chi^n_i}.$
In other words
$ 
\theta^n= \theta^{\gamma_n}_{\chi^n_i}.$ 
Therefore, for all $\ell\in{\cal L}_0,$ we obtain
\begin{equation}\label{aa2}  \int_{X\times A_i} c^{\ell,\gamma_n}_i(x,a_i)\theta^n(dx\times da_i)=
J_i^{\ell,\gamma_n} ([\psi_{-i}^n,\chi^n_i]).
\end{equation}
 By Lemma \ref{A1} in Appendix for every $n\in\mathbb{N}$
 there exists a piecewise constant Markov strategy 
 $\overline{\pi}^n_i$  such that
$$J_i^{\ell,\gamma_n} ([\psi_{-i}^n,\chi^n_i])=J_i^{\ell,\gamma_n} ([\psi_{-i}^n,\overline{\pi}^n_i])
\quad\mbox{for all}\quad \ell\in{\cal L}_0.$$
By (\ref{aa1}) and (\ref{aa2})
$$J_i^{\ell,\gamma_n} ([\psi_{-i}^n,\chi^n_i])\le \kappa_i^\ell<\kappa_i^\ell+\frac{\varepsilon_n}2
\qquad\mbox{for all}\quad \ell\in{\cal L}.$$
Hence, it must hold 
\begin{equation}\label{a4}
J_i^{0,\gamma_n}(\psi^n)\le J_i^{0,\gamma_n}([\psi^n_{-i},\overline{\pi}^n_i])
=J_i^{0,\gamma_n} ([\psi^n_{-i},\chi^n_i]).\end{equation}
We know that
$$
J_i^{\ell,\gamma_n} ([\psi_{-i}^n,\chi^n_i])=
\xi_nJ_i^{\ell,\gamma_n} ([\psi_{-i}^n,\sigma^n_i])+ (1-\xi_n)
J_i^{\ell,\gamma_n} ([\psi_{-i}^n, \phi_i]).
$$
Therefore, by  Lemma \ref{l2}(b) and   (\ref{aa2}), we get
$$\lim_{n\to \infty} J_i^{\ell,\gamma_n} ([\psi_{-i}^n,\chi^n_i])= J_i^{\ell} ([\psi_{-i}^*,\phi_i])$$
for all $\ell\in{\cal L}_0.$  This fact, (\ref{a4}) and Lemma \ref{l2}(a) yield that
$$
J_i^{0}(\psi^*)\le J_i^{0} ([\psi_{-i}^*,\phi_i])=J^0_i([\psi_{-i}^*,\pi_i])
$$
for any feasible strategy $\pi_i\in\Pi_i$ for which (\ref{case2}) holds. 
\hfill $\Box$
\end{proof}

Let $\Psi^\vartheta_i$ be the space of  $\vartheta$-equivalence classes of
strategies in $\Phi_i$ endowed with the weak* topology. Clearly, $\Psi^\vartheta_i$ is 
a compact metric space.  The cost functionals $J^\ell_i(\pmb{\varphi}),$ $\ell \in {\cal L}_0$ and $i\in{\cal N}$, 
are well defined for any profile $\pmb{\varphi}= (\varphi_1,...,\varphi_N)\in
\widehat{\Psi}^\vartheta= \prod_{j\in {\cal N}} \Psi^\vartheta_j.$

\begin{remark}
\label{r8}
From Example 3.16 in \cite{ekm} based on Rademacher's functions, it follows that the 
weak* limit of the sequence of approximate equilibria in Theorem \ref{thm3} need 
not be a stationary equilibrium. The same example can be used to see 
that the cost functionals $J^\ell_i,$ $\ell \in {\cal L}_0$ and $i\in{\cal N}$,
may be discontinuous on $\widehat{\Psi}^\vartheta.$
\end{remark}

Consider the two-person game. It  follows from Lemma \ref{l2} that
$J^\ell_i(\varphi_1,\varphi_2)$ is separately continuous in $\varphi_1$ and $\varphi_2$.
Therefore, the functions
$$
R_1(\varphi_1):=  \min_{\varphi_2\in \Psi_2^\vartheta} 
\max_{\ell\in {\cal L}}\ (
J^\ell_1(\varphi_1,\varphi_2)-\kappa^\ell_1) \quad\mbox{and}\quad
R_2(\varphi_2):=  \min_{\varphi_1\in \Psi_1^\vartheta} 
\max_{\ell\in {\cal L}}\ (J^\ell_2(\varphi_1,\varphi_2)-\kappa^\ell_2)
$$
are upper semicontinuous on $\Psi^\vartheta_1$ and $\Psi^\vartheta_2$, respectively.

\begin{remark}
\label{slater}
Consider a two-person game satisfying the standard Slater condition {\bf A5}. 
Then, it follows  
$$
 R_1(\varphi_1)<0  \quad\mbox{and}\quad R_2(\varphi_2)<0$$
for all $\varphi_1\in \Psi_1^\vartheta$ and  $\varphi_2\in \Psi_2^\vartheta.$ 
Since $R_1$ and $R_2$  are upper semicontinuous on the compact spaces  $\Psi_1^\vartheta$ 
and $\Psi_2^\vartheta$, respectively, we   conclude that
\begin{equation} \label{iinneq}
\max_{\varphi_1\in \Psi_1^\vartheta}  R_1(\varphi_1)<0 \quad\mbox{and}\quad
\max_{\varphi_2\in \Psi_2^\vartheta}  R_2(\varphi_2)<0.
\end{equation} 
Obviously, $\varphi_1$ and $\varphi_2$ in   inequalities 
(\ref{iinneq})  
can be understood as representatives of (denoted by the same letters)  
classes in $\Psi^\vartheta_1$ and $\Psi^\vartheta_2,$ respectively.  Then, it is apparent that {\bf A5} implies {\bf A4} 
for the   considered  two-person game.

Since in the $N$-person $ARAT$ game the cost functionals are continuous
on $\widehat{\Psi}^\vartheta$ with the product topology \cite{dpr2}, 
{\bf A5} implies {\bf A4} in this case.

Finally,we note that in the countable state space case, the weak* topology 
on $\Psi^\vartheta_i$ is actually the topology of point-wise convergence and all
cost functionals $J^\ell_i$ are continuous on   the compact space
$\widehat{\Psi}^\theta$  with the product topology.
Therefore, the standard Slater condition {\bf A4}  
made in the literature for these games,
see \cite{as,ahl,jn3,zhg},  is equivalent to {\bf A5}.
\end{remark}

\section{Non-existence of stationary equilibria in discounted constrained games} \label{s5}

In this section, we consider discounted stochastic games with the given initial state distribution $\eta.$ 
If  $c_i^\ell =0$ and $\kappa_i^\ell =1$ for all
$i\in {\cal N} $ and $\ell\in {\cal L}$, then the game  in this class is 
 trivially constrained  and Assumption {\bf A3} automatically holds.  Our aim is to conclude from \cite{lm} 
that such a game may have no stationary Nash equilibrium. For this, we need the following fact. 

\begin{proposition}
\label{p1}  Let {\bf A1} and {\bf A2} be satisfied and in addition let
$p(\cdot|x,\pmb{a}) \ll \eta$ for all $(x,\pmb{a})\in \mathbb{K}.$  
If $\pmb{\varphi}= (\varphi_1,\ldots,\varphi_N)\in \Phi$ is a stationary Nash equilibrium in 
the discounted stochastic game with the initial state distribution $\eta,$  i.e.,
\begin{equation}
\label{neqm1}
J_i^0(\pmb{\varphi}) \le J^0_i([\pmb{\varphi_{-i}},\pi_i]) 
\end{equation}
for all $i\in {\cal N}$ and $\pi_i\in\Pi_i,$ 
 then   there exists a stationary Nash equilibrium 
$\pmb{\psi}=(\psi_1,...,\psi_N) $  
in the  unconstrained stochastic game for all initial states, i.e.,
\begin{equation}
\label{neqm2}
J_i^0(\pmb{\psi})(x) \le J^0_i([\pmb{\psi_{-i}},\pi_i])(x) 
\end{equation}
for all $i\in {\cal N},$   $\pi_i\in\Pi_i $  and $x\in X.$
Moreover,  $\varphi_i(da_i|x)=\psi_i(da_i|x)$ for $\eta$-a.e. $x\in X $ and for all $i\in {\cal N}.$
\end{proposition}

We start with necessary notation. Let $\pmb{\phi}= (\phi_1,...,\phi_N) \in \Phi.$ Then 
$$\phi(d\pmb{a}|x):=\phi_1(da_1|x)\otimes \phi_2(da_2|x)\otimes\cdots\otimes \phi_N(da_N|x)$$ 
is the product  measure on $A$ determined by
$\phi_i(da_i|x),$ $i=1,2,...,N.$   Recall that by  $\phi_{-i}(d\pmb{a_{-i}}|x)$ 
we denote the projection of  $\phi(d\pmb{a}|x)$
on $A_{-i}.$ We put
$$c_i^0(x,\pmb{\phi}):= \int_A c_i^0(x,\pmb{a})\phi(d\pmb{a}|x)\quad\mbox{and}\quad
p(dy|x,\pmb{\phi}):= \int_A p(dy|x,\pmb{a})\phi(d\pmb{a}|x).$$
If $\sigma_i \in \Phi_i,$ then
$$c_i^0(x,[\pmb{\phi_{-i}},\sigma_i]):= \int_{A_i}\int_{A_{-i}} c_i^0(x,[\pmb{a_{-i}},a_i])
\phi_{-i}(d\pmb{a_{-i}}|x)\sigma_i(da_i|x), $$
$$
p(dy|x,[\pmb{\phi_{-i}},\sigma_i]):= \int_{A_i}\int_{A_{-i}}p(dy|x,
[\pmb{a_{-i}},a_i])\phi_{-i}(d\pmb{a_{-i}}|x)\sigma_i(da_i|x).
$$
If $\nu_i \in \Pr(A_i),$ then
$$c_i^0(x,[\pmb{\phi_{-i}},\nu_i]) := c_i^0(x,[\pmb{\phi_{-i}},\sigma_i]) \quad\mbox{and}\quad  
  p(dy|x,[\pmb{\phi_{-i}},\nu_i]) := p(dy|x,[\pmb{\phi_{-i}},\sigma_i])$$
with $\sigma_i(da_i|x)= \nu_i(da_i)$ for all $x\in X.$

Let $v_i$, $i=1,2,...,N$, be bounded measurable functions on $X.$
For each $x\in X,$ by $\Gamma_x(v_1,...,v_N)$ we denote the one-step $N$-person game, 
where the payoff (cost) function for player $i\in {\cal N} $  is
$$(1-\alpha)c_i^0(x,\pmb{a}) + \alpha\int_X v_i( y)p(dy|x,\pmb{a}),
\quad\mbox{where}\quad  
\pmb{a}=(a_1,...,a_N)\in A.$$ 

\begin{proof}{\it of Proposition \ref{p1}}  
From (\ref{neqm1}), it follows that for each set $S\in {\cal F},$ we have
\begin{eqnarray*}
J_i^0(\pmb{\varphi}) &=&
\int_{  X}
\left((1-\alpha)c^0_i(x,\pmb{\varphi}) +\alpha
\int_XJ_i^0(\pmb{\varphi})(y)p(dy|x,\pmb{\varphi})\right)\eta(dx)\\
&\le&
\int_S\min_{\nu_i\in\Pr(A_i)}\left((1-\alpha)c^0_i(x,[\pmb{\varphi_{-i}},\nu_i]) +\alpha
\int_XJ_i^0(\pmb{\varphi})(y)p(dy|x,[\pmb{\varphi_{-i}},\nu_i])\right)\eta(dx)\\   
&+&\int_{X\setminus S}
\left((1-\alpha)c^0_i(x,\pmb{\varphi}) +\alpha
\int_XJ_i^0(\pmb{\varphi})(y)p(dy|x,\pmb{\varphi})\right)\eta(dx)
\end{eqnarray*}
Hence, for each $S\in {\cal F},$ 
\begin{eqnarray*}
\lefteqn{ \int_{  S}
\left((1-\alpha)c^0_i(x,\pmb{\varphi}) +\alpha
\int_XJ_i^0(\pmb{\varphi})(y)p(dy|x,\pmb{\varphi})\right)\eta(dx) 
\le}\\& &
\int_S\min_{\nu_i\in\Pr(A_i)}\left((1-\alpha)c^0_i(x,[\pmb{\varphi_{-i}},\nu_i]) +\alpha
\int_XJ_i^0(\pmb{\varphi})(y)p(dy|x,[\pmb{\varphi_{-i}},\nu_i])\right)\eta(dx).
\end{eqnarray*}
Thus, for every $i\in {\cal N},$ there exists $S_i\in {\cal F}$ 
such that $\eta(S_i)=1$ and for all $x\in S_i,$ we have
\begin{eqnarray}
\label{neqm3}
\lefteqn{ (1-\alpha)c^0_i(x,\pmb{\varphi}) +\alpha
\int_XJ_i^0(\pmb{\varphi})(y)p(dy|x,\pmb{\varphi})
 \le }\\ && \nonumber  
 \min_{\nu_i\in\Pr(A_i)}\left((1-\alpha)c^0_i(x,[\pmb{\varphi_{-i}},\nu_i]) +\alpha
\int_XJ_i^0(\pmb{\varphi})(y)p(dy|x,[\pmb{\varphi_{-i}},\nu_i])\right).
\end{eqnarray}
Let $\widehat{S}:= S_1\cap S_2\cdots\cap S_N.$ Now consider the game
$\Gamma_x(v_1,...,v_N),$ where $v_i(y)= J_i^0(\pmb{\varphi})(y),$ $y\in X.$  
By Lemma 5 in \cite{nr}, there exists   $\pmb{\phi} \in \Phi$ such that 
$\pmb{\phi} (d\pmb{a}|x) = (\phi_1(da_1|x),...,\phi_N(da_N|x))$ is a Nash equilibrium in the game 
$\Gamma_x(v_1,...,v_N)$ for all  $x\in X\setminus \widehat{S}.$ For every $i\in {\cal N},$ define
$\psi_i(da_i|x):= \varphi_i(da_i|x),$ if $x\in \widehat{S},$ and 
$\psi_i(da_i|x):= \phi_i(da_i|x),$ if $x\in X\setminus \widehat{S}.$
Then, using (\ref{neqm3}), we conclude that
$\pmb{\psi}(d\pmb{a}|x) = (\psi_1(da_1|x),...,\psi_N(da_N|x))$ 
is a Nash equilibrium in the game $\Gamma_x(v_1,...,v_N)$ 
for all $x\in X.$ Define
$v_i^0(y):= v_i(y)= J_i^0(\pmb{\varphi})(y)$ for each $y\in \widehat{S}$ and
$$ v_i^0(y):= 
(1-\alpha)c^0_i(y,\pmb{\psi}) +\alpha
\int_XJ_i^0(\pmb{\varphi})(z)p(dz|y,\pmb{\psi})
$$ 
for each $y\in X\setminus \widehat{S}.$ Then,  $\eta(X\setminus \widehat{S})=0$ and  our assumption
$p(\cdot|x,\pmb{a})\ll \eta(\cdot),$  $(x,\pmb{a})\in \mathbb{K},$  imply that 
$\Gamma_x(v_1^0,...,v_N^0)= \Gamma_x(v_1,...,v_N)$ for all
$x\in X.$   Therefore, for all $x\in X,$  $\psi(d\pmb{a}|x)$ is a Nash equilibrium 
in the game  $\Gamma_x(v_1^0,...,v_N^0)$   and 
$$
v^0_i(x)=(1-\alpha)c^0_i(x,\pmb{\psi}) +\alpha
\int_Xv_i^0(y)(y)p(dy|x,\pmb{\psi}).
$$ 
Using these facts and the Bellman equations for discounted dynamic programming \cite{bs,hll}, 
we conclude that (\ref{neqm2}) holds.
\hfill $\Box$\end{proof}

\begin{remark}
\label{r7}  Levy and McLennan \cite{lm}
gave an example of a discounted stochastic game with no constraints 
having no stationary Nash equilibrium. This is an $8$-person stochastic game
with finite action sets for the players and $X=[0,1]$ as the state space. 
The definitions of payoff functions and transition probabilities in their game are rather complicated 
and are not given here. We only mention that  the transition probabilities are absolutely continuous 
with respect to the probability measure
$\eta_1=(\lambda_1 +\delta_1)/2,$ where $\lambda_1$ is the Lebesgue measure 
on $[0,1]$ and $\delta_1$ is the Dirac measure concentrated at the point $1.$
Assume that $\eta_1$  is the initial state distribution in this game.
If this game had a stationary Nash equilibrium, then
by Proposition \ref{p1}, it would have  a stationary Nash equilibrium for all initial states.  
From Levy \cite{lm},  it follows it  is impossible.\footnote{We thank John Yehuda Levy for pointing out this fact.}
\end{remark}

\section{Remarks on games with unbounded costs}
\label{s6}

Our results can be extended to a class of games with unbounded
cost functions $c_i^{\ell}$ under some uniform integrability condition
introduced in \cite{fjn}. The method for doing this relies on truncations of the  
costs and using an approximation by bounded games. This was done in our paper 
\cite{jn3} in the countable state space case.
In a special situation, described below and inspired by the work of Wessels
\cite{w} on dynamic programming, a reduction to the bounded case  can be obtained
by the well-known data transformation as described in Remark 2.5 in \cite{dpr1} or Section 10 in \cite{fp}.  
Following Wessels \cite{w}, we make the following assumptions.\\
		
\noindent{\bf  Assumption W} \\ 
{\bf (i)}   There exist a measurable function $\omega:X\to [1,\infty)$ and $c_0>0$ such that   
$|c^\ell_i(x,\pmb{a})|\le c_0\omega(x)$   for all $x\in X,$ $\pmb{a}\in A,$ $i\in {\cal N}$ and $\ell\in {\cal L}_0.$\\
{\bf (ii)}    There exists $\beta>1$ such that $\alpha \beta <1$ and
$$\int_X\omega(y)p(dy|x,\pmb{a}) \le \beta\omega(x)$$
for all $x\in X,$ $\pmb{a}\in A.$ \\ 
{\bf (iii)}   If $\pmb{a}^n \to \pmb{a}$ as  $n\to \infty,$ then 
$$
\int_X|\delta(x,y,\pmb{a}^n) - \delta(x,y,\pmb{a})|\omega(y)\mu(dy) \to 0.
$$\\

To describe the equivalent  model with bounded costs we extend the state space $X$ by adding an isolated 
{\it absorbing state} $0^*.$ All the costs at this absorbing state are {\it zero}.
Let
$c_i^{\ell,\omega}(x,\pmb{a}):= \frac{c_i^\ell(x,\pmb{a})}{\omega(x)},$ and
$$p^\omega(B|x,\pmb{a}):= \frac{\int_B \omega(y)p(dy|x,\pmb{a})}{\beta\omega(x)}, \quad B \in {\cal F},\ x\in X,\ \pmb{a}\in A,$$
$$p^\omega (0^*|x,\pmb{a}): = 1 -\frac{\int_X \omega(y)p(dy|x,\pmb{a})}{\beta\omega(x)},\quad x\in X,\ \pmb{a}\in A.$$
Now define the new initial state distribution as 
$$\eta_0(B):= \frac{\int_B\omega(x)\eta(dx)}{\eta\omega},\quad \mbox{where}\quad \eta\omega = \int_X\omega(x)\eta(dx).$$ 
Here, we assume that $\eta\omega<\infty.$ 
Then, we obtain  primitive data for a bounded constrained stochastic game,
in which the discount factor is $\alpha \beta.$   We denote the
expected discounted costs in the bounded game under consideration 
by ${\cal J}^\ell_i(\pmb{\pi}).$ 
It is easy to see that
$$
{\cal J}^\ell_i(\pmb{\pi})= \frac{J^\ell_i(\pmb{\pi})}{\eta\omega },\quad\mbox{for all}\quad i\in {\cal N}, 
\ell\in {\cal L}_0,\ \pmb{\pi}\in \Pi.
$$
Theorems \ref{thm1} and  \ref{thm3}  can be established for the  
bounded game described above with minor modifications. 
For example, one has to define new constraint constants as $\kappa^\ell_i/\eta\omega,$
$i\in {\cal N}, \ell\in {\cal L}.$  Using the above transformation, we can immediately 
deduce similar results for games with unbounded cost functions satisfying Assumption {\bf W}. 

\section{Appendix}
\label{s7}

In this section, we prove a lemma which plays an important role in the proofs of our theorems. 

Let player $i \in {\cal N}$ be fixed. 
We also fix $\gamma>0$, the partition ${\cal P}^\gamma =\{X^\gamma_n: n\in \mathbb{N}_0 \}$ 
of the state space $X,$  the  cost functions $c^{\ell,\gamma}_i$
and the transition function $p^\gamma$ in the game ${\cal G}^\gamma.$ 
We fix $\pmb{\varphi_{-i}}\in  \Phi_{-i}^\gamma=\prod_{j\in {\cal N}\setminus \{i\}}\Phi_j^\gamma.$

A piecewise constant Markov strategy for player $i$ is a sequence $\pi_i= (f^t)_{t\in\mathbb{N}},$ 
where $f^t\in\Phi_i^\gamma$ for all $t\in\mathbb{N}.$ \\

\begin{lemma} \label{A1}  For fixed $\pmb{\varphi}\in \Phi^\gamma$ and  each $\phi_i\in \Phi_i$ there  
exists a  piecewise constant Markov strategy 
$\pi_i= (f^t)_{t\in\mathbb{N}}$  for player $i$ such that  
$$
J_i^{\ell,\gamma}([\pmb{\varphi_{-i}},\phi_i])= 
J_i^{\ell,\gamma}([\pmb{\varphi_{-i}},\pi_i])\quad \mbox{for all}\quad \ell\in{\cal L}_0.
$$
\end{lemma} 

For a proof we need some auxiliary results. Let $d\in\mathbb{N}.$ 

\begin{lemma} \label{A2}  Assume that $Y\in {\cal F}$ and
$\rho_0$ is a probability measure on $X $  such that $\rho_0(Y)=1.$ 
Let $v=(v_0,...,v_{d-1})$, where every $v_j:X\to \mathbb{R}$ is a bounded  measurable function. 
Then, there exist points $y_0,...,y_{d} \in Y$ and non-negative numbers 
$\beta_0,...,\beta_{d}$ such that $\sum_{j=0}^{d} \beta_j =1$ and 
\begin{equation}
\label{cov}
\int_Y v(x)\rho_0(dx)=  \sum_{j=0}^{d}\beta_j v(y_j).
\end{equation}\end{lemma}

\begin{proof}  Consider the distribution function of $v$ defined by: 
$\zeta_v(B) :=\rho_0(v^{-1}(B))$, where $B$   is any Borel set  in $\mathbb{R}^d.$  
Using Theorem  16.13 on page  229 in \cite{b} and Lemma 3 on page 74 in \cite{fer}, we obtain
$$ \int_Y v(x)\rho_0(dx)= \int_{\mathbb{R}^d}z\zeta_v(z)dz \in  co\{v(y): y\in X\}.$$
Applying Carath{\'e}odory's theorem, we find points  $y_0,...,y_{d}\in Y$ 
and numbers $\beta_0,...,\beta_{d} \ge 0$ such that 
$\sum_{j=0}^{d} \beta_j =1$ and (\ref{cov}) holds. \hfill $\Box$
\end{proof}

We use  ${\cal C}(A_i)$ to denote the space of all real-valued continuous functions on $A_i $ and $\Pr(A_i)$ 
for the space of all probability measures on $A_i.$ 

\begin{lemma} \label{A3} 
Let $\rho$ be a probability measure on $X$. 
For each $\ell\in{\cal L}_0$ assume that $u^\ell:X\times A_i \to\mathbb{R}$ 
is a bounded function such that $u^\ell(x,a_i) = u_n^\ell(a_i)$ for all 
$x\in X^\gamma_n,$ $a_i\in A_i,$ where 
$u^\ell_n \in {\cal C}(A_i)$,  $n\in \mathbb{N}_0.$
Then, for any $\phi_i\in \Phi_i$ there exists $f\in \Phi^\gamma_i$  such that  
\begin{equation}
\label{ruf}
\int_X\int_{A_i} u^\ell(x,a_i)\phi_i(da_i|x)\rho(dx)=
\int_X\int_{A_i} u^\ell(x,a_i)f(da_i|x)\rho(dx) \quad\mbox{for all } \ell\in{\cal L}_0.
\end{equation}\end{lemma}

\begin{proof}  Assume first that $\rho(X_n^\gamma)>0$ and define 
$\rho_0(B) =\frac{\rho(B\cap X_n^\gamma)}{\rho(X_n^\gamma) },$ 
$B \in {\cal F}.$ Applying Lemma \ref{A2} with
$d=L+1$ and $v=(u^0,...,u^{L}),$ we infer that  
there exist points $y_0(n),...,y_{L+1}(n)$ in $X_n^\gamma$ 
and $\beta_0(n),...,\beta_{L+1}(n) \ge 0$ such that 
$\sum_{j=0}^{L+1}\beta_j(n)=1$ and
\begin{eqnarray*}
\frac{1}{\rho(X_n^\gamma)}\int_{X_n^\gamma}\int_{A_i}u^\ell(x,a_i)\phi_i(da_i|x)\rho(dx)&=&
\frac{1}{\rho(X_n^\gamma)}\int_{X_n^\gamma}\int_{A_i}u_n^\ell(a_i)\phi_i(da_i|x)\rho(dx)\\ 
&=& \sum_{j=0}^{L+1}\beta_j(n)
\int_{A_i}u^\ell_n(a_i)\phi_i(da_i|y_j(n))\ \mbox{ for all }  \ell\in{\cal L}_0.
\end{eqnarray*}
For each $x\in X_n^\gamma$, define $f(da_i|x) :=\nu_n(da_i),$ where $\nu_n 
\in \Pr(A_i)$ is given as
$$ \nu_n(da_i):= 
\sum_{j=0}^{L+1}\beta_j(n)
 \phi_i(da_i|y_j(n)).
$$
If $\rho(X_n^\gamma)=0,$ then $f(da_i|x)$ is defined for all
$x \in X_n^\gamma$ by  $f(da_i|x)=\nu_n(da_i)$
where $\nu_n$ is any fixed measure in $\Pr(A_i).$
Note that, we have 
$$
\int_{X_n^\gamma}\int_{A_i}u^\ell(x,a_i)\phi_i(da_i|x)\rho(dx)=
  \int_{A_i}u_n^\ell(a_i)\nu_n(da_i )\rho(X_n^\gamma)= 
\int_{X_n^\gamma}\int_{A_i}u^\ell(x,a_i)f(da_i|x)\rho(dx), 
$$
for all $\ell\in{\cal L}_0,\  n\in \mathbb{N}_0.$
Hence,
$$ \sum_{n\in\mathbb{N}_0}
\int_{X_n^\gamma}\int_{A_i}u^\ell(x,a_i)\phi_i(da_i|x)\rho(dx) =
\sum_{n\in\mathbb{N}_0}
\int_{X_n^\gamma}\int_{A_i}u^\ell(x,a_i)f(da_i|x)\rho(dx), 
$$
for all $\ell\in{\cal L}_0,$  which implies (\ref{ruf}). 
\hfill $\Box$ \end{proof}

Since $i\in {\cal N},$ $\gamma>0$, 
$\pmb{\varphi_{-i}}\in \Phi_{-i}^\gamma$ and $\phi_i \in \Phi_i$ 
are fixed, the notation  for the proof of Lemma \ref{A1} can be simplified.

Let $\varphi_{-i}(d\pmb{a_{-i}}|x)$ be the product measure  on
$A_{-i}$ induced by $\varphi_j(da_j|x)$ with $j\not=  i.$ 
For $\ell \in {\cal L}_0,$ $x\in X$ and $a_i\in A_i,$ we put 
$$
c^\ell(x,a_i):= \int_{A_{-i}}c_i^{\ell,\gamma}(x,[\pmb{a_{-i}},a_i])
\varphi_{-i}(d\pmb{a_{-i}}|x),$$
$$
q(dy|x,a_i):=   \int_{A_{-i}}p^{ \gamma}(dy|x,[\pmb{a_{-i}},a_i])
\varphi_{-i}(d\pmb{a_{-i}}|x).
$$
Next, we put
$$
c^\ell_{\phi_i}(x) := \int_{A_i}c^\ell(x,a_i)\phi_i(da_i|x), $$
and, for any  bounded measurable function $w:X\to \mathbb{R},$ 
$$Q_{\phi_i}w(x):= \int_{A_i}w(y)q(dy|x,a_i)\phi_i(da_i|x).  
$$
Similarly, we define
$c^\ell_{g}(x)$ and $Q_{g}w(x)$
for any $g\in \Phi_i^\gamma.$ 
Next, if $g^1,g^2,...,g^T \in \Phi_i^\gamma,$ then
$$\eta w =\int_Xw(x)\eta(dx)\quad\mbox{and}\quad
Q_{g^1}Q_{g^2}\cdots Q_{g^T}w(x) = Q_{g^1}(Q_{g^2}\cdots Q_{g^T}w)(x)$$
and   $$\eta Q_{g^1}Q_{g^2}\cdots Q_{g^T}w :=
\int_X Q_{g^1}Q_{g^2}\cdots Q_{g^T}w(x)\eta(dx).$$
Note that $\eta Q_{g^1}Q_{g^2}\cdots Q_{g^T}$ is the probability distribution of the 
state  $x_{T+1}$ of the process, when player $i$ uses a Markov strategy
$(g^t)_{t\in\mathbb{N}}.$
	
We now introduce new notation for expected costs.
Recalling that $\phi_i\in\Phi_i,$  we put  
$$
I^\ell(\phi_i)(x) := J_i^{\ell,\gamma}([\pmb{\varphi_{-i}},\phi_i])(x)\quad\mbox{and}\quad 
I^{\ell,\eta}(\phi_i) := \int_X I^\ell(\phi_i)(x)\eta(dx),\quad \ell\in {\cal L}_0.
$$
If $\pi_i=(g^t)_{t\in\mathbb{N}}$ is a piecewise constant
strategy for player $i,$ then $I^{\ell,\eta}_T(\pi_i)=
I^{\ell,\eta}_T(g^1,...,g^T)	$ denotes the expected discounted cost in the
$T$-step game ${\cal G}^\gamma$ under assumption that the other players use $\pmb{\varphi_{-i}}.$
Then, the cost over the infinite time horizon is
$$I^{\ell,\eta}(\pi_i)= \lim_{T\to\infty} I^{\ell,\eta}_T(\pi_i).$$

\begin{proof}{\it of Lemma \ref{A1}} We  show by induction that for given $\phi_i\in \Phi_i$ there exists 
$\pi_i=(f^t)_{t\in\mathbb{N}}$ with $f^t\in \Phi_i^\gamma$
for all $t\in \mathbb{N}$ such that for all $T\in\mathbb{N},$ we have 
\begin{equation}
\label{receqn}
I^{\ell,\eta}(\phi_i)= I^{\ell,\eta}(f^1,...,f^T) + 
\alpha^T \eta Q_{f^1}\cdots Q_{f^T}((1-\alpha)c^\ell_{\phi_i}
+\alpha Q_{\phi_i}I^{\ell}(\phi_i)).
\end{equation}
We shall use the following equation
$$ I^\ell(\phi_i)(x) = (1-\alpha)c^\ell_{ \phi_i}(x) + \alpha 
Q_{\phi_i}I^\ell(\phi_i)(x),\quad \mbox{for each}\quad x\in X.$$
Assume that $T=1.$ 	Then,	
\begin{eqnarray*}
 I^{\ell,\eta}(\phi_i) &=&  \eta((1-\alpha) c^\ell_{\phi_i}
+\alpha Q_{\phi_i}I^{\ell}(\phi_i))\\ &=&
\int_X\int_{A_i}\left((1-\alpha)c^\ell(x,a_i) + \alpha\int_X
I^{\ell}(\phi_i)(y)q(dy|x,a_i)\right)\phi_i(da_i|x)\eta(dx).
\end{eqnarray*}
Applying Lemma \ref{A3} with  $\rho=\eta$ and 
\begin{equation}
\label{ugiven}
u^\ell(x,a_i)=(1-\alpha)
c^\ell(x,a_i) + \alpha\int_X
I^{\ell}(\phi_i)(y)q(dy|x,a_i)
\end{equation}
we obtain $f^1\in \Phi_i^\gamma$ such that   
$$
\int_X\int_{A_i}u^\ell(x,a_i)\phi_i(da_i|x)\eta(dx)= 
\int_X\int_{A_i}u^\ell(x,a_i) f^1(da_i|x)\eta(dx) \quad\mbox{ for all}\quad \ell\in{\cal L}_0.$$
Then, we get
\begin{eqnarray*}
 I^{\ell,\eta}(\phi_i) &=& \eta I^\ell(\phi_i) = \eta((1-\alpha) c^\ell_{\phi_i}
+\alpha Q_{\phi_i}I^{\ell}(\phi_i))\\ &=&
\eta  ((1-\alpha) c^\ell_{f^1}
+\alpha   Q_{f^1}I^{\ell}(\phi_i)) = 
\eta  (1-\alpha) c^\ell_{f^1} + \alpha \eta Q_{f^1}I^{\ell}(\phi_i) \\ &=&
I_1^{\ell,\eta}(f^1) + \alpha	\eta  Q_{f^1}((1-\alpha)c^\ell_{\phi_i}+\alpha
Q_{\phi_i} I^{\ell}(\phi_i)) \quad\mbox{fro all}\quad \ell \in{\cal L}_0.
\end{eqnarray*} 
We have obtained (\ref{receqn}) for $T=1.$ Assume now that (\ref{receqn}) 
holds for $T=m$ with some $m\ge 1.$ 
Then we have for some $f^1,...,f^m \in \Phi^\gamma_i$ that 
$$
I^{\ell,\eta}(\phi_i)= I^{\ell,\eta}(f^1,...,f^m) + 
\alpha^m \eta Q_{f^1}\cdots Q_{f^m}((1-\alpha)c^\ell_{\phi_i}
+\alpha Q_{\phi_i}I^{\ell}(\phi_i))
$$
for all $\ell\in{\cal L}_0.$
Applying Lemma \ref{A3} with $u^\ell(x,a_i)$ given by (\ref{ugiven}) 
and $\rho= 		
\eta Q_{f^1}\cdots Q_{f^m},$ we obtain $f^{m+1} \in \Phi_i^\gamma$ such that 
\begin{eqnarray*}	
\lefteqn{\eta Q_{f^1}\cdots Q_{f^m}((1-\alpha)c^\ell_{\phi_i}
+\alpha Q_{\phi_i}I^{\ell}(\phi_i))  = 
\eta Q_{f^1}\cdots Q_{f^m}((1-\alpha)c^\ell_{f^{m+1}}
+\alpha Q_{f^{m+1}}I^{\ell}(\phi_i))}\\ &=&
\eta Q_{f^1}\cdots Q_{f^m} (1-\alpha)c^\ell_{f^{m+1}}+
\alpha \eta Q_{f^1}\cdots Q_{f^m}Q_{f^{m+1}}((1-\alpha)
c^\ell_{\phi_i} + \alpha Q_{\phi_i}
I^{\ell}(\phi_i)).
\end{eqnarray*}
Thus for all $\ell\in{\cal L}_0$ we get
\begin{eqnarray*}
I^{\ell,\eta}(\phi_i) &=& I^{\ell,\eta}(f^1,...,f^m) + 
\alpha^m 
\eta Q_{f^1}\cdots Q_{f^m} (1-\alpha)c^\ell_{f^{m+1}}
\\ &&+ \alpha^{m+1}
\eta Q_{f^1}\cdots Q_{f^m}Q_{f^{m+1}}((1-\alpha)
c^\ell_{\phi_i} + \alpha Q_{\phi_i}
I^{\ell}(\phi_i)) \\ &=&
I^{\ell,\eta}(f^1,...,f^{m+1}) +
\alpha^{m+1}
\eta Q_{f^1}\cdots Q_{f^m}Q_{f^{m+1}}((1-\alpha)
c^\ell_{\phi_i} + \alpha Q_{\phi_i}
I^{\ell}(\phi_i)).
\end{eqnarray*}
This finishes the induction step. Taking the limit in (\ref{receqn}) as $T\to \infty$, we obtain
$$
I^{\ell,\eta}(\phi_i) =  I^{\ell,\eta}(\pi_i) \quad\mbox{with}\quad \pi_i= (f^1,f^2,...)	$$
for all $ \ell \in{\cal L}_0.$		
Going back to our original notation, we deduce that this is the assertion of Lemma \ref{A1}. \hfill $\Box$
\end{proof}

 \noindent
{\bf Acknowledgement. } We thank  two reviewers for very helpful reports. 
We acknowledge the financial support from the National Science Centre, Poland: Grant 2016/23/B/ST/00425.\\

\end{document}